\documentclass[11pt]{amsart}

\usepackage{t1enc}
\usepackage{amsthm,amssymb}
\usepackage{latexsym}
\usepackage{amsmath} 
\usepackage{graphics}
\usepackage{epsfig,multicol}
\usepackage{amsfonts}
\usepackage[latin1]{inputenc}
\usepackage[english]{babel}
\usepackage{ mathrsfs }
\usepackage{enumerate}
\usepackage[toc,page]{appendix}

\newtheorem{theorem}{Theorem}[subsection]
\newtheorem{proposition}{Proposition}[subsection]
\newtheorem{lemma}{Lemma}[subsection]
\newtheorem{definition}{Definition}[subsection]

\newtheorem{diagram}{Figure}[subsection]

\def\hpic #1 #2 {\mbox{$\begin{array}[c]{l} \epsfig{file=#1,height=#2}
\end{array}$}}
 
\def\vpic #1 #2{\mbox{$\begin{array}[c]{l} \epsfig{file=#1,width=#2}
\end{array}$}}

\newcommand  {\rmn}\romannumeral

\newcommand {\ds}{\mathscr{D}}
\newcommand{\h}{\mathcal H}
\newcommand{\e}{\mathcal E}
\newcommand{\1}{\mathcal I}
\newcommand{\5}{\vskip 5pt}

\begin{document}
\newcommand{\BA}{\vpic{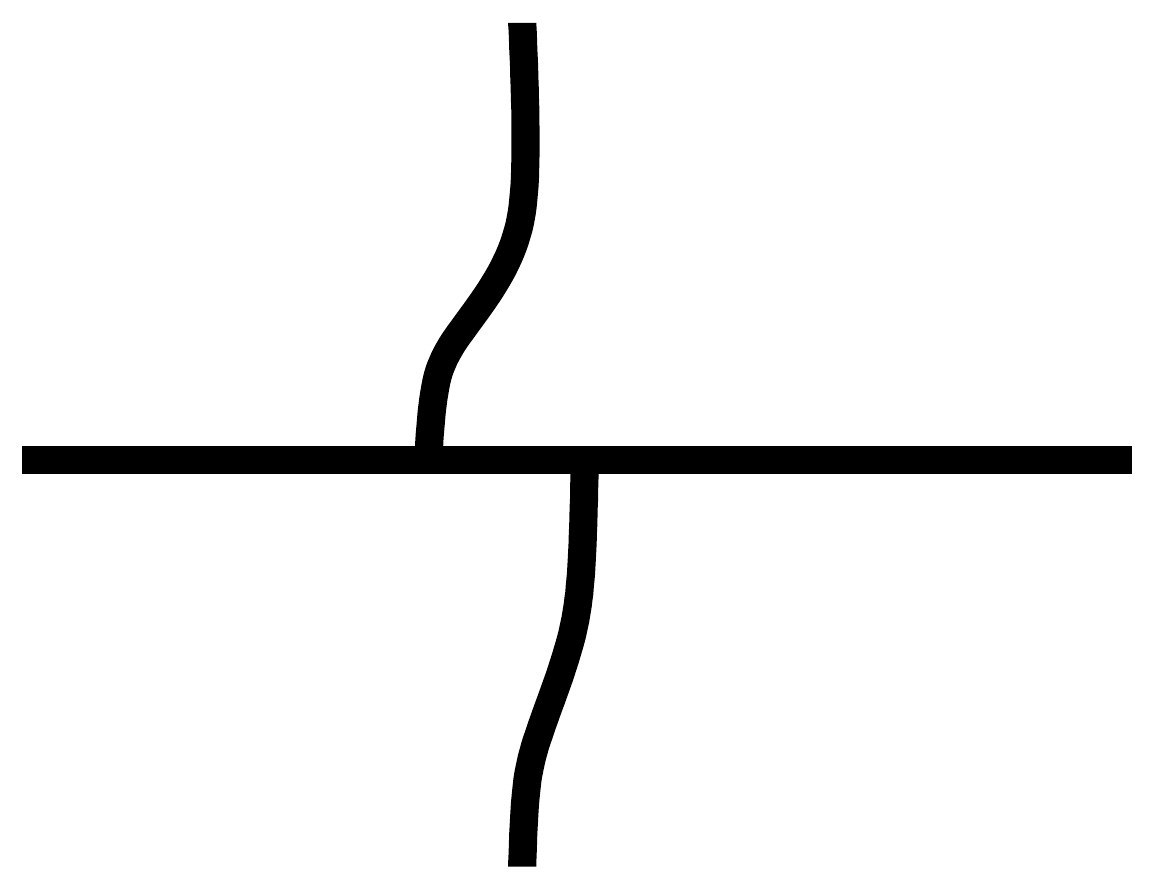} {0.2in}}
\newcommand{\BB}{\vpic{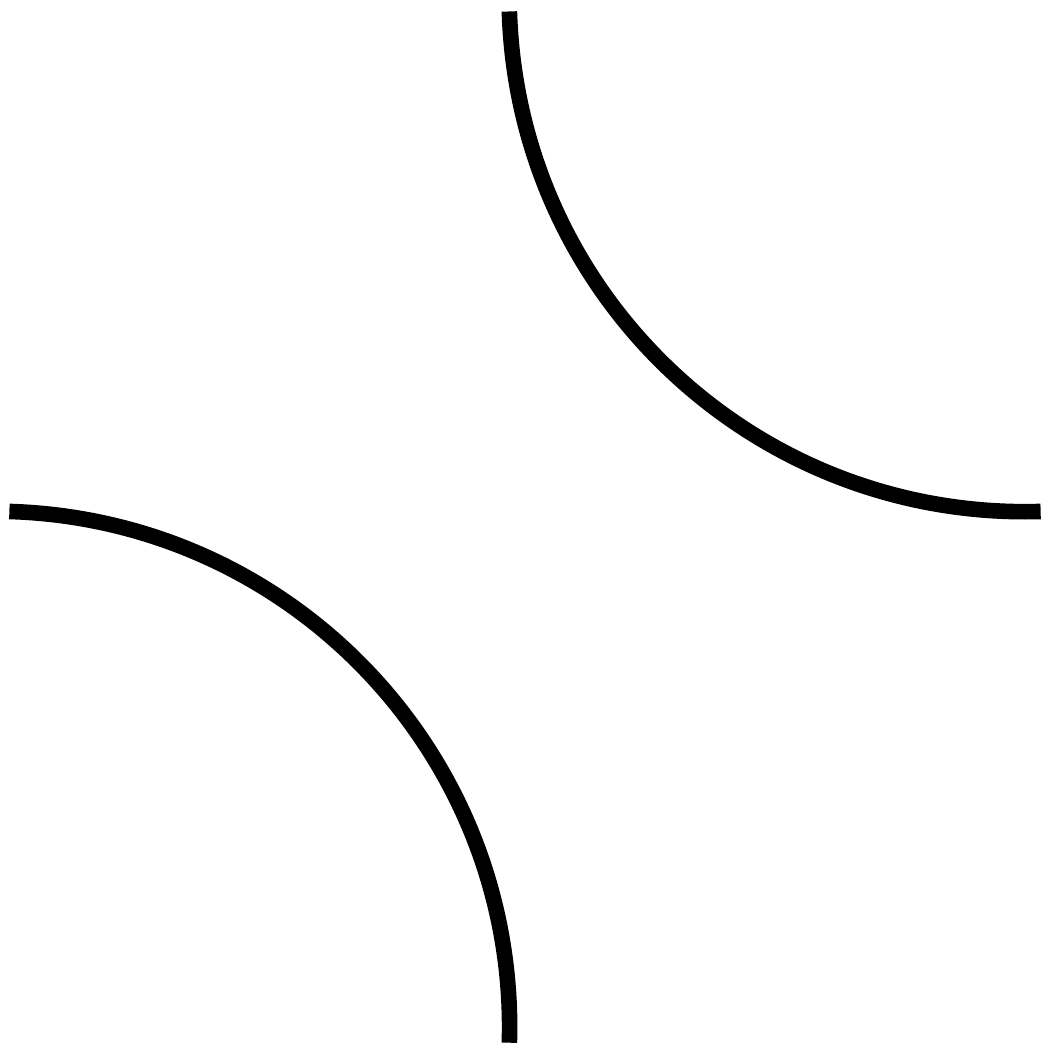} {0.2in}}
\newcommand{\BC}{\vpic{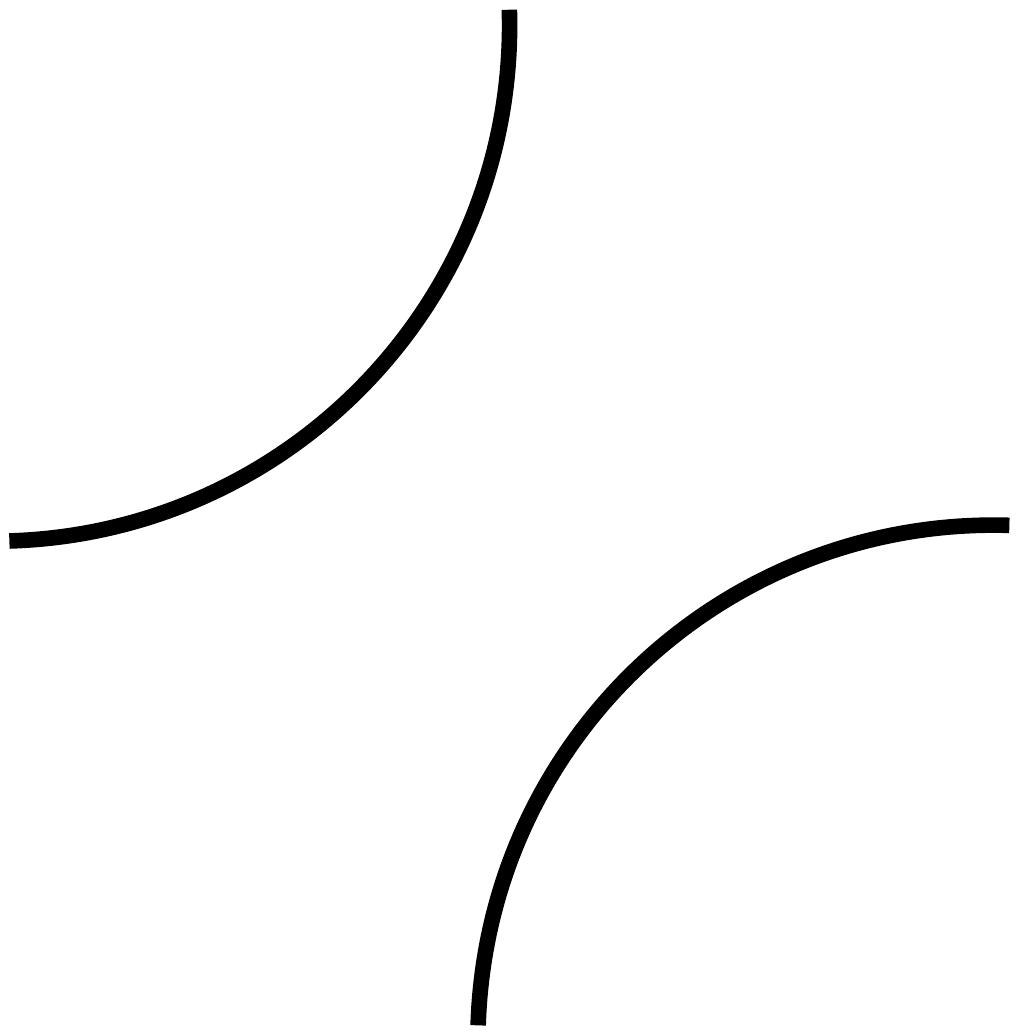} {0.2in}}
\title{A no-go theorem for the continuum limit of a periodic quantum spin chain.}
\author{Vaughan F. R. Jones}
\thanks{V.J. is supported by the NSF under Grant No. DMS-0301173 and grant DP140100732, Symmetries of subfactors}
\maketitle
\begin{abstract}
We show that the Hilbert space formed from a block spin renormalization construction of a cyclic quantum spin chain (based on
the Temperley-Lieb algebra) does not support a  chiral conformal field theory whose Hamiltonian generates translation on the circle
as a continuous limit of the rotations on the lattice. 
\end{abstract}

\section{Introduction}
This paper is part of an ongoing effort to construct a conformal field theory for every finite index subfactor in such a way
that the standard  invariant of the subfactor, or at least its quantum double, can be recovered from the CFT.

In \cite{jo4} an infinite dimensional Hilbert space, envisioned as a limit of the Hilbert spaces of finite quantum spin chains on the circle, was
constructed using the following data:

1) A  (positive definite) ``planar algebra''  $P$ (\cite{jo2}) together with an affine unitary representation of it (\cite{jo3},\cite{JR}). 

(An affine unitary representation is $\mathbb N$-graded and the $nth$ graded component is thought of as the Hilbert space of a period quantum
spin chain with $n$ spins. For the simplest planar algebra, the $nth$ graded component is just $\otimes^n(\mathbb C^2)$ so it is literally
the Hilbert space of a quantum spin chain. In \cite{J22} it is argued that planar algebras are indeed physically meaningful generalisations of ordinary
spin chains.)

2) An element  $R$ of $P_4$ with the normalization property

\begin{diagram} \label{normalisation}
\qquad \qquad \vpic{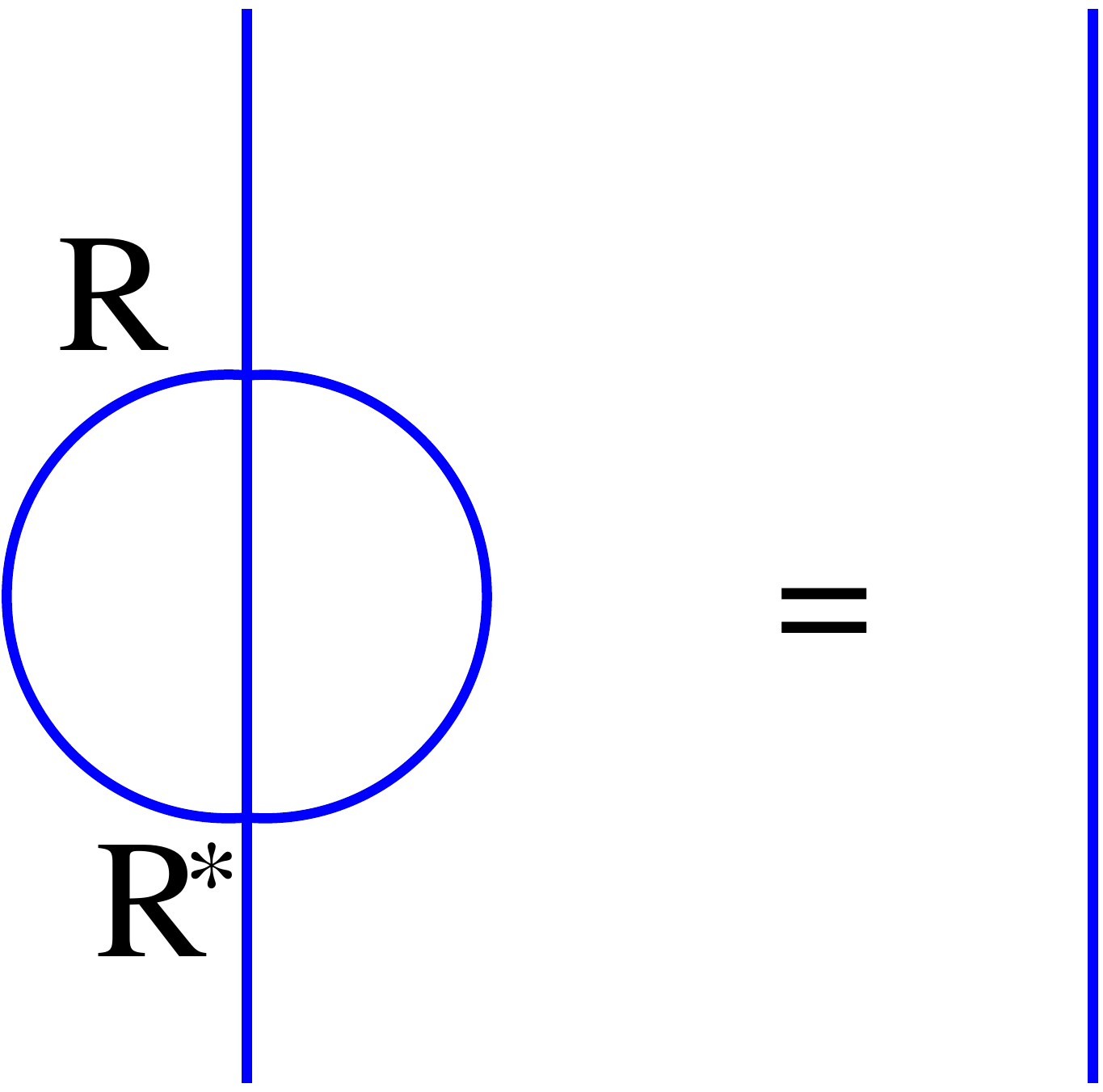} {1.5in}
\end{diagram}

(See the appendix for an explanation of planar algebra. But one does not need to know planar algebras to understand the constructions. Just interpret the R's inside the pictures 
as tensors with indices on the strings and the picture as giving a scheme for contracting indices-\cite{Pen}. 
 This is already common usage in the physics literature.)

The element $R\in P_4$ serves as a way of embedding the Hilbert space of a quantum spin chain with $n$ spins into
the Hilbert space of a spin chain with $3n$ spins. As a result of conversation with Tobias Osborne and Guifre Vidal we shall call these limit Hilbert spaces "semicontinuous limits" 
of Hilbert spaces for the quantum spin chains.

We will begin in the next section by giving a simplified and more general version of the construction of the semicontinuous limit
of \cite{jo4}. For the circular version this will give us unitary representations of Thompson's group $T$ which acts by \emph{local scale transformations}.
This representation was hoped to tend to a representation of $Diff^+(S^1)$  by taking limits of elements of $T$ on the semicontinuous limit.
In particular the rotation group $Rot(S^1)$  was hoped to arise as the closure of the dyadic rotations is $T$. 

This approach is somewhat naive and very open to criticism on physical grounds, and
in this paper we show that this possibility fails as dramatically as possible, at least for one example of a semicontinuous limit $\mathcal H$. We show in
fact that in this case for any two vectors $\xi,\eta \in \mathcal H$, $$\lim_{n\rightarrow \infty} \langle \rho_{\frac{1}{2^n}}\xi,\eta\rangle = 0$$
where $\rho_x$ is unitary on $\mathcal H$ representing rotation of the circle $\mathbb R / \mathbb Z$ by a dyadic rational $x$.
Thus even in the weak topology the rotations by dyadic rationals are discontnuous (though we do not show that 
$\lim_{r\rightarrow 0}\langle \rho_{r}\xi,\eta\rangle = 0$).

Faced with this failure there are two possibilities. The first is to abandon the semicontinuous limit and look for other ways to obtain the 
Hilbert space of the conformal field theory. One idea which is relatively close to our approach is to replace our embeddings of quantum
spin chains one in another by Evenbly and Vidal's MERA (see \cite{ev}) which introduces more local interactions between the spins. Vidal's numerical 
evidence could be interpreted as saying that the Hilbert space obtained by the MERA embeddings should naturally support a CFT.
We have not made any progress along these lines. See also \cite{CV}.

Another possibility is to redefine the goal. After all, the direct limit approach does  produce states of a quantum spin chain that 
transform according to local scale transformations of the lattice. Perhaps this semicontinuous limit is of value in the analysis of
critical behaviour of lattice quantum spin chains. The mathematics is completely different from that of CFT but the structure of
the nogo theorem certainly yields numerical data that could be relevant, e.g. the rate at which $ \langle \rho_{\frac{1}{2^n}}\xi,\eta\rangle$ 
tends to zero. And the appearance of the \emph{transfer matrix} in the proof  is oddly dual to the role of the transfer matrix in
models which, if \cite{ps} is to be believed, should have CFT as a scaling limit. In these solvable models the Hamiltonian-the infinitesimal
generator of time evolution-is obtained as the logarithmic derivative of the transfer matrix with respect to the spectral parameter.
In our case the infinitesimal behaviour of time (=space) evolution is governed by the transfer matrix.

In a future paper we will investigate scale invariant  Hamiltonians and transfer matrices on the semicontinuous limit.

\emph{In this paper all planar algebras will be \it{unshaded} and all representations will have a positive definite
invariant inner product unless otherwise specified. (See the appendix for the meaning of this terminology.)}

\section{A categorical construction of the Thompson groups.}
\subsection{A group of fractions for certain categories.}\label{categories}
The following construction of groups is well known and goes back at least as far as a 1931 result on semigroups of
Ore. (See also the work \cite{ddgkm} in the category context.)
The use of direct limits and functors  to construct \emph{representations} of groups of fractions is probably also well known, but less so as the corresponding representations of Thompson's groups seem to  have appeared first in \cite{jo4}. For this reason we give a self-contained exposition of the
whole business. The extension from group of fractions to groupoid of fractions is clear.

Let $\mathfrak K$ be a small category with the following 3 properties.\\
\begin{enumerate}[(i)] \label{properties}
\item (Unit)\quad There is an element $1\in Ob(\mathfrak K)$ with $Mor_{\mathfrak K}(1,a)\neq \emptyset$ for all $a\in Ob(\mathfrak K)$.
\item (Stabilisation) Let $\displaystyle \mathcal D=\underset{a\in Ob(\mathfrak K)}\cup Mor_{\mathfrak K}(1,a)$. Then for each $f,g\in \mathcal D$ there are morphisms $p$ and $q$ with $pf=qg$. 

\item (Cancellation) \quad If $pf=qf$ for $f\in \mathcal D$ then $p=q$.

\end{enumerate}
\begin{proposition}\label{directsystem} If we define $\preceq$ on $\mathcal D$ by $f\preceq g$ iff $g=pf$ for some morphism $p$ then $\mathcal D$ becomes
a directed set. Moreover given a functor $\mathfrak F$ from $\mathfrak K$ to some category $\mathfrak C$ then the  sets $A_f$, for $f\in \mathcal D$,
$$A_f=Mor_{\mathfrak C}(\Phi(1),\Phi(target(f)))$$ 
together with the maps $\iota_f^g:A_f\rightarrow A_g$ when $f\preceq g$ and $g=pf$ given by
$$\iota_f^g(v)=\Phi(p)\circ v$$
form a direct system  denoted $A(\Phi)$.
\end{proposition}

\begin{proof} The proof is just verification of the axioms of directed set and direct system from the properties of $\mathfrak K$.
(Note that $\preceq$ is not necessarily a partial order, just a preorder.) For instance the directed set property follows from stabilisation.
\end{proof}

We will explore the direct limit $\underset {\rightarrow}\lim A(\Phi)_f$. 
Recall that the direct limit $\underset{\rightarrow}\lim A_i$ of a direct system is by definition the disjoint
union of the $A_i$ (which we will call $\mathcal P$) modulo the equivalence relation $\cong$ defined by
$x\in A_i\cong y\in A_j \iff \exists k$ with $i\preceq k, j\preceq k$ and $\iota_i^k(x)=\iota_j^k(y)$ .
 If the $\iota$'s are injections then each $\Phi(\mathcal S)$ is naturally identified with a
subset of $\underset{\rightarrow}\lim \Phi(\mathcal S)$.

First we take the functor $\Phi$ to be the identity functor $\mathfrak I$ from $\mathfrak K$ to itself.
By definition then the direct limit $\underset {\rightarrow}\lim A({\mathfrak I})_f$ is the quotient of
the set $\mathcal P$ of all ordered pairs $(f,g)$ with $f,g\in \mathcal D$  by the equivalence relation
$(f_1,g_1)\cong (f_2,g_2) \iff \exists p,q \in \mathfrak K$ such that  $(pf_1,pg_1)=(qf_2,qg_2)$:
$$\underset {\rightarrow}\lim A({\mathfrak I})_f=\mathcal P /\cong$$
Now given two elements  $(f_1,g_1)$ and $(f_2,g_2)$ in $\mathcal P$ we can choose by stabilisation
morphisms $p,q\in \mathfrak K$ with $pg_1=qf_2$. Then define $$(f_1,g_1)_{(p,q)}(f_2,g_2)=(pf_1,qg_2).$$
\begin{proposition} The map from $\mathcal P \times \mathcal P \rightarrow \underset {\rightarrow}\lim A({\mathfrak I})_f$
taking $((f_1,g_1),(f_2,g_2))$ to $[(f_1,g_1)_{(p,q)}(f_2,g_2)]$ depends neither on the choice of $(p,q)$ nor on the
choices of $(f_1,g_1)$ and $(f_2,g_2)$ in their $\cong$ equivalence classes. The resulting operation makes 
$\underset {\rightarrow}\lim A({\mathfrak I})_f$ into a group.
\end{proposition} \label{group}
\begin{proof} The proof follows in a relatively routine manner from stabilization and cancellation. The identity element is
$[(1,1)]$ and the inverse of $[(f,g)]$ is $[(g,f)]$. We will have to do all
the details of well-definedness again to prove the next result so we leave the rest of the proof to be checked then.
\end{proof}
\begin{definition} The group defined by the previous proposition will be called the group of fractions $G_{\mathfrak K}$   of $\mathfrak K$.
\end{definition}
If $\Phi$ is not the identity functor we obtain an action of $G_{\mathfrak K}$ on  $\underset {\rightarrow}\lim A({\Phi})_f$.

The direct limit $\underset {\rightarrow}\lim A({\Phi})_f$ is the quotient of
the set $\mathcal Q$ of all ordered pairs $(f,g)$ with $f\in \mathcal D$ and $g\in Mor_{\mathfrak C}(\Phi(1), target(\Phi(f)))$  by the equivalence relation
$(f_1,g_1)\cong (f_2,g_2) \iff \exists p,q \in \mathfrak K$ such that  $(pf_1,\Phi(p)\circ g_1)=(qf_2,\Phi(q)\circ g_2)$:
$$\underset {\rightarrow}\lim A({\Phi})_f=\mathcal Q /\cong$$

Now given an element  $(f_1,g_1)\in \mathcal P$ (as in proposition \ref{group}) and $(f_2,g_2)$ in $\mathcal Q$ we can choose by stabilisation
morphisms $p,q\in \mathfrak K$ with $pg_1=qf_2$. Then define $$(f_1,g_1)_{(p,q)}(f_2,g_2)=(pf_1,\Phi(q)\circ g_2).$$

\begin{proposition}\label{action}
The map from $\mathcal P \times \mathcal Q \rightarrow \underset {\rightarrow}\lim A({\Phi})_f$
taking $((f_1,g_1),(f_2,g_2))$ to $[(f_1,g_1)_{(p,q)}(f_2,g_2)]$ depends neither on the choice of $(p,q)$ nor on the
choices of $(f_1,g_1)$ and $(f_2,g_2)$ in their $\cong$ equivalence classes. The resulting operation
defines an action of $G_{\mathfrak K}$ on 
$\underset {\rightarrow}\lim A({\Phi})_f$.\:
$$(f_1,g_1)((f_2,g_2))=(pf_1,\Phi(q)\circ g_2).$$
 If the category $\mathfrak C$ is linear, the action of $G_{\mathfrak K}$ is
linear and if moreover the Hom spaces of $\mathfrak C$ are Hilbert spaces and the $\Phi(f)$ are isometries then
$\underset {\rightarrow}\lim A({\Phi})_f$ is a pre-Hilbert space and the action of $G_{\mathfrak K}$ is unitary. 
Each individual Hilbert space $(f, \Phi(target(f)))$ is a Hilbert subspace of $\underset {\rightarrow}\lim A({\Phi})_f$ and hence 
its Hilbert space completion.
\end{proposition}
\begin{proof}
First suppose $p$ and $q$ are changed to $p'$ and $q'$. Then by stabilisation there are $r$ and $s$ such 
that $sp'g_1=rpg_1$. So by cancellation $$sp'=rp.$$
Moreover $rpg_1=rqf_2$ and $sp'g_1=sq'f_2$ hence $rqf_2=sq'f_2$ and by cancellatioin
$$rq=sq'$$
Thus $(pf_1,\Phi(q)g_2)\cong (rpf_1,\Phi(rq)g_2)=(sp'f_1,\Phi(sq')g_2)\cong(p'f_1,\Phi(q')g_2)$.
To see the action property (or associativity of the group operation), let
$(f_1,g_1), (f_2,g_2) \in \mathcal P$ and $(f_3,g_3)\in \mathcal Q$ be given. Choose $r, s \in\mathfrak K$ with
$rg_2=sf_3$ and $p,q$ such that $pg_1=qrf_2$. Then to calculate $[(f_1,g_1)]([(f_2,g_2)]([(f_3,g_3)]))$ and
$([(f_1,g_1)][(f_2,g_2)])([(f_3,g_3)])$ we can, by well-definedness, use $(pf_1,pg_1), (qrf_2,qrg_2)$ and $(qsf_3,\Phi(qs)g_3)$
instead. Then both expressions yield $[(pf_1,\Phi(qs)g_3)]$.
The assertions about linearity and unitarity are trivial.
\end{proof}
Examples of groups and representations constructed in this way are more or less interesting depending on how ``small'' the category
is compared to the group it produces. We list a few examples below where our point of view brings nothing new. The first example shows that the construction is universal in some sense but of no
interest at all in this case.
 
 \begin{enumerate}[(i)]
 \item  (All groups)\quad If G is a group, consider it as a small category $\mathcal G$ with one element. It trivially satisfies the conditions  of \ref{properties} and of course $G_{\mathcal G}=G$. The representations 
 obtained are just the usual group actions.

 \item (Fundamental group)\quad If $X$ is a path-connected space and $\mathcal G$ its homotopy groupoid then if we choose $1$ to be any point of $X$ the 
 properties of \ref{properties} are trivially satisfied and one obtains $\pi_1(X)$ from the construction. If the target category $\mathcal C$ for $\Phi$ is Vect  and 
 we are given a flat connection on $X$ then $\Phi$ can be constructed by parallel transport and one obtains the holonomy representation of $\pi_1(X)$.
 \item(Integers and rationals)\quad  If we take $\mathbb N\cup \{0\}$ with addition we obtain $\mathbb Z$. No new representations will be obtained in this way. A functor $\Phi$  to sets is given by the image of $1$ which is simply a transformation $T$ of the set (image under $\Phi$ of
 the object of $\mathbb N$). If $T$ is invertible then the map $(n,x)\mapsto T^{-n}(x)$ defines a $\mathbb Z$-equivariant 
 bijection from $ \underset {\rightarrow}\lim\Phi$ to $X$. If $T$ is not invertible things are more complicated. (For instance for
 the identity functor.) We leave it to the reader to work out the answer in general but we observe that if $T$ is a linear 
 transformation of a finite dimensional vector space $V$ then $ \underset {\rightarrow}\lim\Phi$ is 
 $$W=\underset{n\in \mathbb N}\cap T^nV$$
 on which $T$ acts surjectively hence invertibly by what we call $T_\infty$. The isomorphism takes a pair $(p,v)$ in $ \underset {\rightarrow}\lim\Phi$ to $T_\infty^{-n}v$ provided $n$ is sufficiently large that $T^nv\in W$.

 \item(Braids)\quad See also \cite{ddgkm}. If we take the category $\mathcal B_n$ consisting of positive braids on $n$ strings, with one object,
 then Garside theory shows that any braid $a$ is of the form $\Delta^k b$ for some positive braid $b$ and $k\leq 0$
 where $\Delta$ is the positive half twist braid. So if $p,q\in \mathcal B_n$ then $q=qp^{-1}p$. Writing $qp^{-1}$ as
 $\Delta^kb$ we see $bp=\Delta^{-k}q$. This shows that stabilisation holds in $\mathcal B_n$.
 Cancellation is obvious and it is clear we get $G_{\mathcal B_n}=B_n$. If $\mathcal B_n$ is represented by invertible
 matrices they define a functor to Vect and we get a representation of $B_n$. We have not fully analyzed the situation
 when the matrices representing $\mathcal B_n$ are not all invertible.


  \end{enumerate} 

We now turn to examples of categories $\mathcal F$ which  we  will use to obtain genuinely interesting representations of
$G_{\mathcal F}$.

\subsection{The category of planar forests and Thompson's group $F$.}\label{planar forest}
By "forest" we will mean a planar binary forest whose roots lie on a horizontal line and whose leaves lie on another 
horizontal line above the roots. Two such forests will be considered the same if they can be isotoped one to another
in the obvious way. Here is an example of a forest:\\

\qquad \qquad\qquad \vpic{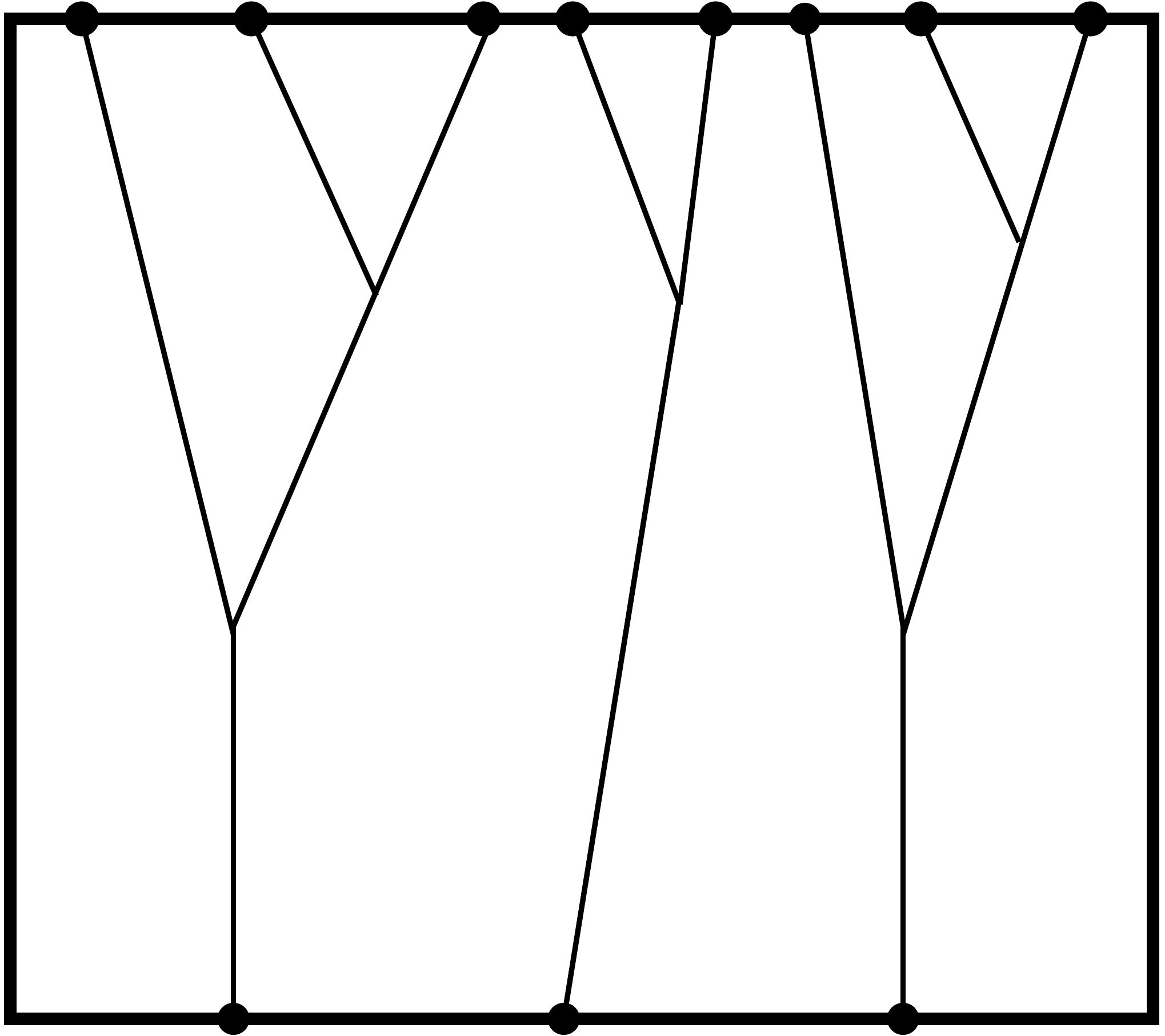} {2in} 

Forests form a category $\mathfrak F$ whose objects are $\mathbb N$ and whose  morphisms from $m$ to $n$ are
the forests with $m$ roots and $n$ leaves. Obviously $Mor_{\mathfrak F}(m,n)$ is only non-empty if $m\leq n$. Composition
of morphisms is just the obvious stacking of planar forests. Clearly $Mor_{\mathfrak F}(1,n)$ is the set of all planar binary rooted trees with
$n$ leaves. So  $\cup _n Mor_{\mathfrak F}(1,n)$ is the directed set $\mathfrak T$ of all such trees.
It is obvious that, for $\mathcal T, \mathcal S\in \mathfrak T$ there is at most one morphism $\mathcal F\in \mathfrak F$ with
$\mathcal T\circ \mathcal F=\mathcal S$. Moreover any binary planar tree can be completed to a full binary tree with $2^m$ leaves
for some large $m$. Thus the category $\mathfrak F$ satisfies the conditions of \ref{properties} and we may form the
group $G_{\mathfrak F}$.

\begin{proposition} The group $G_{\mathfrak F}$ is isomorphic to Thompson's group $F$ of piecewise linear 
homeomorphisms of $[0,1]$.
\end{proposition}
\begin{proof} See \cite{CFP} for an explanation of how elements of $F$ can be represented by pairs of binary trees, up to 
a certain equivalence relation. Check that this is the same as our definition of $G_{\mathfrak F}$.
\end{proof}

The construction generalizes immediately to  the categories $\mathfrak F_m$ of planar rooted forests all of whose vertices are $n+1$-valent.
The groups $G_{\mathfrak F_m}$ are the Thompson groups $F_n$ (not free groups!) where $2$ is replaced
by $n$ in the definitions.

\subsection{The category of annular forests and Thompson's group $T$.}

Any elegance this treatment has derives from the paper \cite{gl} of Graham and Lehrer.

\begin{definition} We define a \rm{rooted, affine binary forest}  $F_{m,n}$ to be a planar binary  forest which can be drawn in the
strip $\mathbb R\times [0,1] \subset \mathbb R^2$ as a diagram with $m$ roots in the open interval $(0,1)$ and which is invariant under
horizontal translation by $\mathbb Z$. The subforest connected to the roots in $(0,1)$ is to have its $n$ leaves on $\mathbb R\times
\{1\}$ we may suppose none of the leaves has an integral $x$-coordinate.
\end{definition}
Here is a picture of an element $F_{3,7}$:

\hspace{1.5in} \vpic {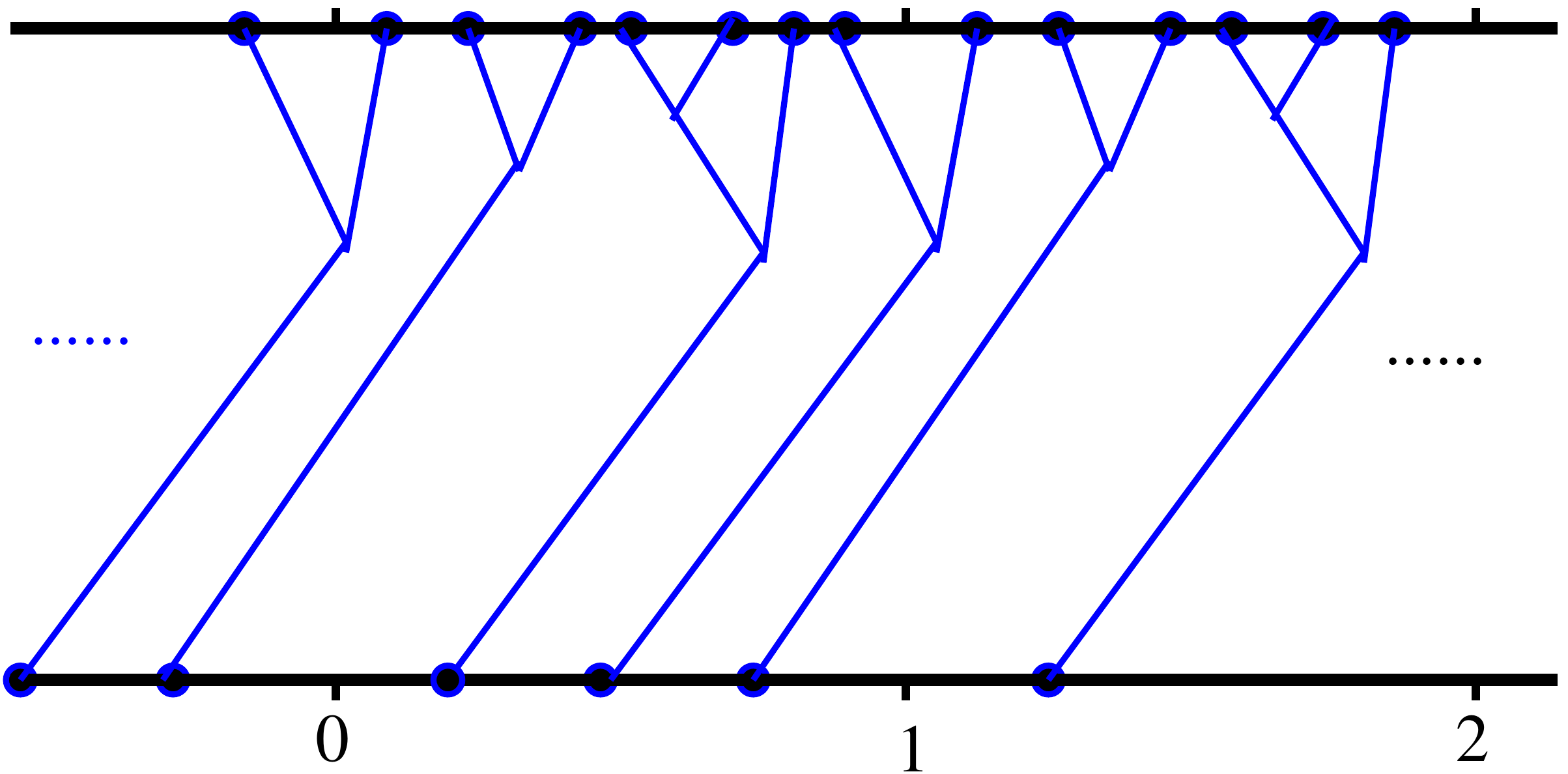} {3in}

Where the diagram is continued to the left and right by periodicity.

The forest below in $Mor_ {\mathfrak {AF}(n,n)}$ will be called $\rho_n$ (illustrated with $n=4$):
\vskip 5pt
\hspace {1.5in} \vpic {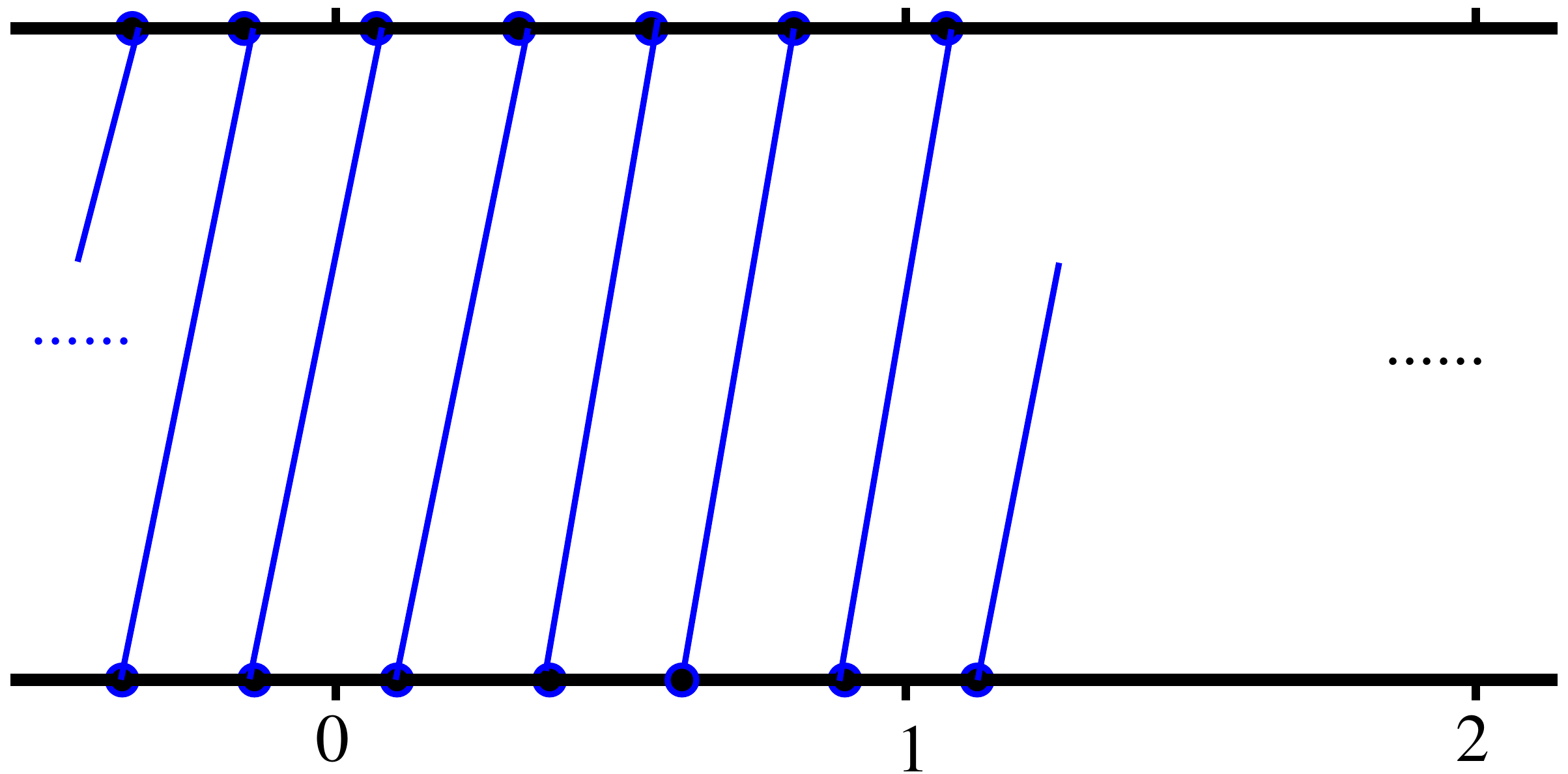} {3in}
 
 And we will set $\tau_n=\rho_n^n$.
  
By planarity and peridocity the leaves of the subforest connected to the roots in $(0,1)$ lie in an interval of length $1$. There
are exactly $m$ roots in any interval of length $1$ on the $x$-axis and exactly $n$ leaves in any interval of length $1$ 
on the line  $\mathbb R\times
\{1\}$.

Rooted affine binary forests may be stacked by lining up the leaf of the bottom forest with smallest positive $x$ coordinate
with the root on the top forest with smallest positive $x$ coordinate. Planarity dictates how all the other roots match up with
leaves. (Alternatively one could insist that all the roots and leaves lie on specific points so they will line up automatically.)
Thus the set $\mathfrak {AF}$ of all rooted affine binary forests forms a category. Note that $Mor_{\mathfrak {AF}}(m,m)$ is a group
isomorphic to $\mathbb Z$ generated by $\rho_n$ and every $F_{m,n}$ can be composed with an element of this group so that it becomes
a forest inside $[0,1]\times[0,1]$ extended to the whole strip by periodicity.
With this observation it is clear that $\mathfrak {AF}$ satisfies the conditions of \ref{properties} with the number $1$ as the object $1$.
We see moreover that the data of an element  $F_{m,n}\in\mathfrak {AF}$ is the same as a planar forest $F$ as in \ref{planar forest}  and 
a unique integer $k$ so that $\rho_n^k\circ F=F_{m,n}$. (In the example of a forest we have given above, $k=3$ but nothing is to stop it 
being bigger than $n$.)

The group $G_ {\mathfrak {AF}} $ is a natural central extension $\tilde T$ of Thompson's group $F$ which may be defined as piecewise
linear periodic foliations of the strip which are smooth except at dyadic rationals and whose lines in the smooth parts have
slope a power of $2$. 

To obtain Thompson's group $T$ on the nose from a category construction we take quotient of 
 $\mathfrak {AF}$ by the action of $\mathbb Z$  which acts on $Mor_{ \mathfrak {AF}}(m,n)$ by composing with powers of $\tau_n$.
  That this action is compatible with composition follows from the simple relation
  $$ \tau_n\circ F_{m,n}=F_{m,n} \circ \tau_m^{-1}$$  So one obtains a category $\mathfrak T$  with the same objects as
  $\mathfrak {AF}$ and finitely many morphisms for each $n$, one for each pair 
  (an element of $\mathfrak F$  with $n$ leaves,an element of $\mathbb Z/n\mathbb Z$).
  (such a pair obviously represents an orbit under the action of $\rho^n$). Thus an element of $G_{\mathfrak T}$ is an equivalence
  class of pairs of rooted binary trees  with the same number $n$ of leaves, one of them marked for each tree. Using a power of
  $\rho_n$, one of the marks can be taken to be the leftmost leaf. Comparing with \cite{CFP} we see we have obtained Thompson's group $T$.

  As for $F$, one can replace $2$ by any larger integer to get affine categories whose groups of fractions are the 
  Thompson groups $T_n$.
  
  \subsection{Thompson's group $V$ and the braided Thompson groups.}
  Thompson's group $V$ is a larger group than $F$ which allows discontinuous piecewise linear maps of the
  circle that swap the intervals on which an element is linear. Thus any element is given by a pair of binary planar  rooted trees
  together with a permutation of the leaves of one of them which determines how the intervals are to be identified.
  We can capture this group with our category method by letting $\mathfrak V$ be the category whose objects are $\mathbb N$ and
  whose morphisms are pairs consisting $(\mathcal F, \pi)$ where $\mathcal F \in \mathfrak F$ and $\pi$ is a 
  permutation of the leaves of $\mathcal F$. The permutations themselves are morphisms in $\mathfrak V$ and the key 
  observation is that for each $\pi\in Mor_{\mathfrak V}(m,m)$ and each $F\in Mor_{\mathfrak F}(m,m)$, there is
  a $\sigma\in  Mor_{\mathfrak V}(n,n)$ with $$\sigma\circ \mathcal F=\mathcal F\circ \pi.$$
  We leave it to the reader to make sense of this and how it yields a well defined category structure on
  $\mathfrak V$ whose group of fractions is $V$. 
  
  For the braided Thompson $BF$ group the situation is very similar, the category $\mathfrak {BF}$ consists of pairs  $(\mathcal F,\alpha)$ where
  $\alpha$ is an $n-string$ braid where $n$ is the number of leaves of $\mathcal F$. See \cite{deh} for the definitions of braided Thompson
  groups.
  
  \section{Review and development of the action of the Thompson groups  on the semicontinuous limit.}
\subsection{How to obtain representations, unitary and otherwise.}\label{construction}

The previous section would be no more than a curiosity  were it not for the
fact that we can now mechanically and uniformly construct actions of all the Thompson groups using
functors.

We will use the language of planar algebras (see the appendix) but we would like to make it clear once again that \emph{all the essential
ideas and many interesting examples are exhibited in the tensor planar algebra so one needs to understand no more
about planar algebras than how diagrams specify ways to contract tensors.}

The common ingredient is a *-planar algebra $\mathcal P=(P_n)$ and an element $R$ of $P_{n+1}$ where we are dealing with
Thompson groups relevant to rationals of the form $\frac{a}{n^k}$ with $a$ and $k$ integers.
The element $R$ must satisfy the ``unitarity'' condition: \\

\hspace{1.7in} \vpic {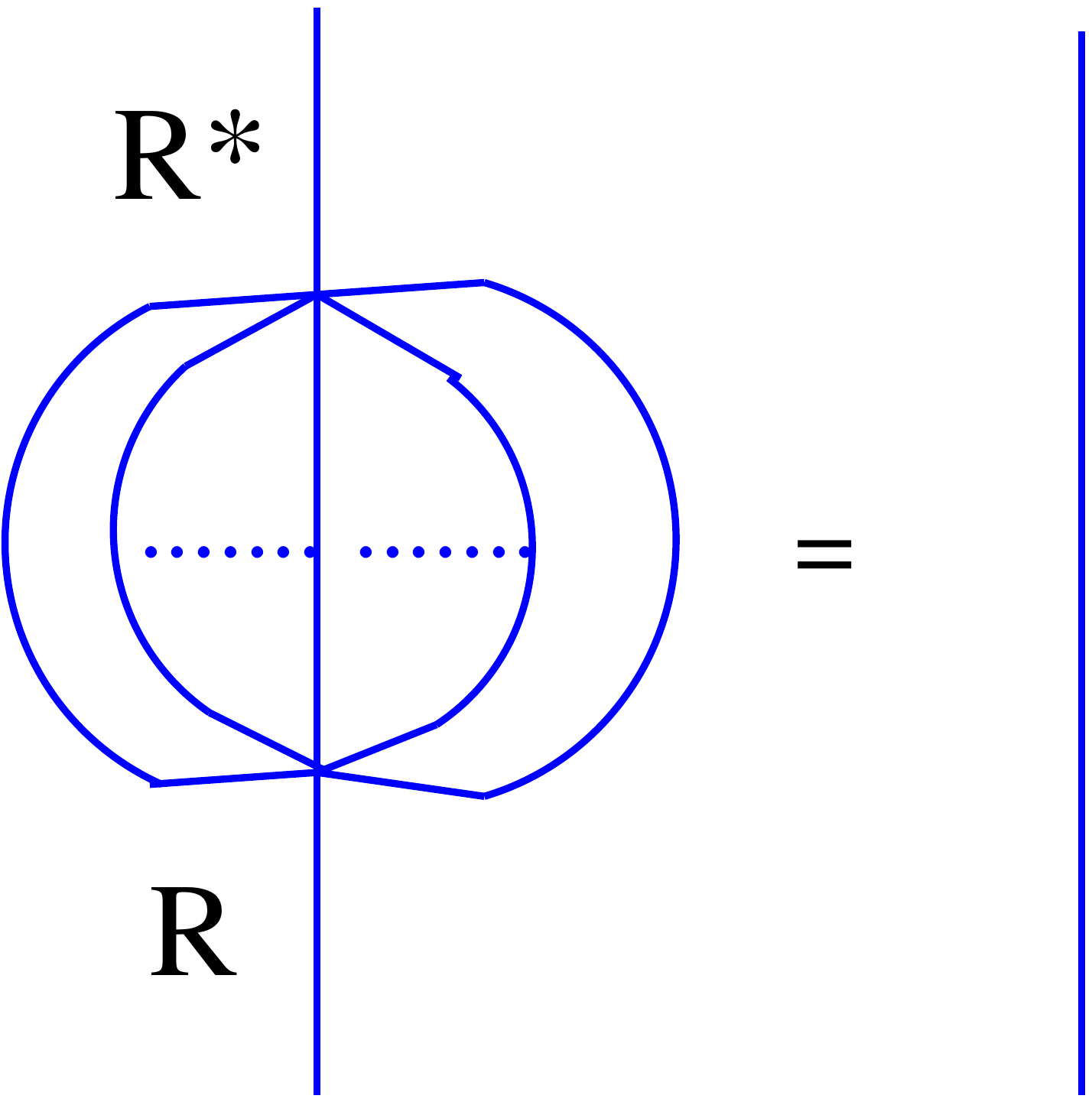} {2in}

Where there are $n$ strings joining the discs containing $R$ and $R^*$.
Such an $R$ will give rise to representations of the various forest categories of section \ref{categories}. We fix $\mathcal P$ and $R$ and treat each case individually as there are some caveats.

\begin{enumerate}[(i)]
\item{Thompson's group $F$.}
Here $R$ is in $P_3$. Let $\mathcal V=(V_n)$ be a representation of the rectangular category $\mathfrak R$ of $\mathcal P$ (In the
case of tensors, the objects are just the tensor powers of a fixed vector space , with morphisms being tensors mapping between the different tensor powers). A morphism in $Mor_{\mathfrak R}(m,n)$ is just a rectangle with $m$  marked points on the bottom and $n$ on the top, filled with elements of $\mathcal P$ in discs connected
by strings among themselves and to the marked points on the boundary.

\begin{definition} Let $\mathcal F$ be a rooted planar forest in $Mor_{\mathfrak F}(m,n)$. Define
 $\Phi_R(\mathcal F)$ to be the element of  $Mor_{\mathfrak R}(m,n)$ obtained by replacing every vertex in $\mathcal F$
 by a disc containing a copy of $R$ as follows:
 \vskip 5pt
 
 \qquad $\mathcal F=$ \vpic {forestnew} {1.5 in} \hspace{0.4in} $\Phi(\mathcal F)=$ \vpic {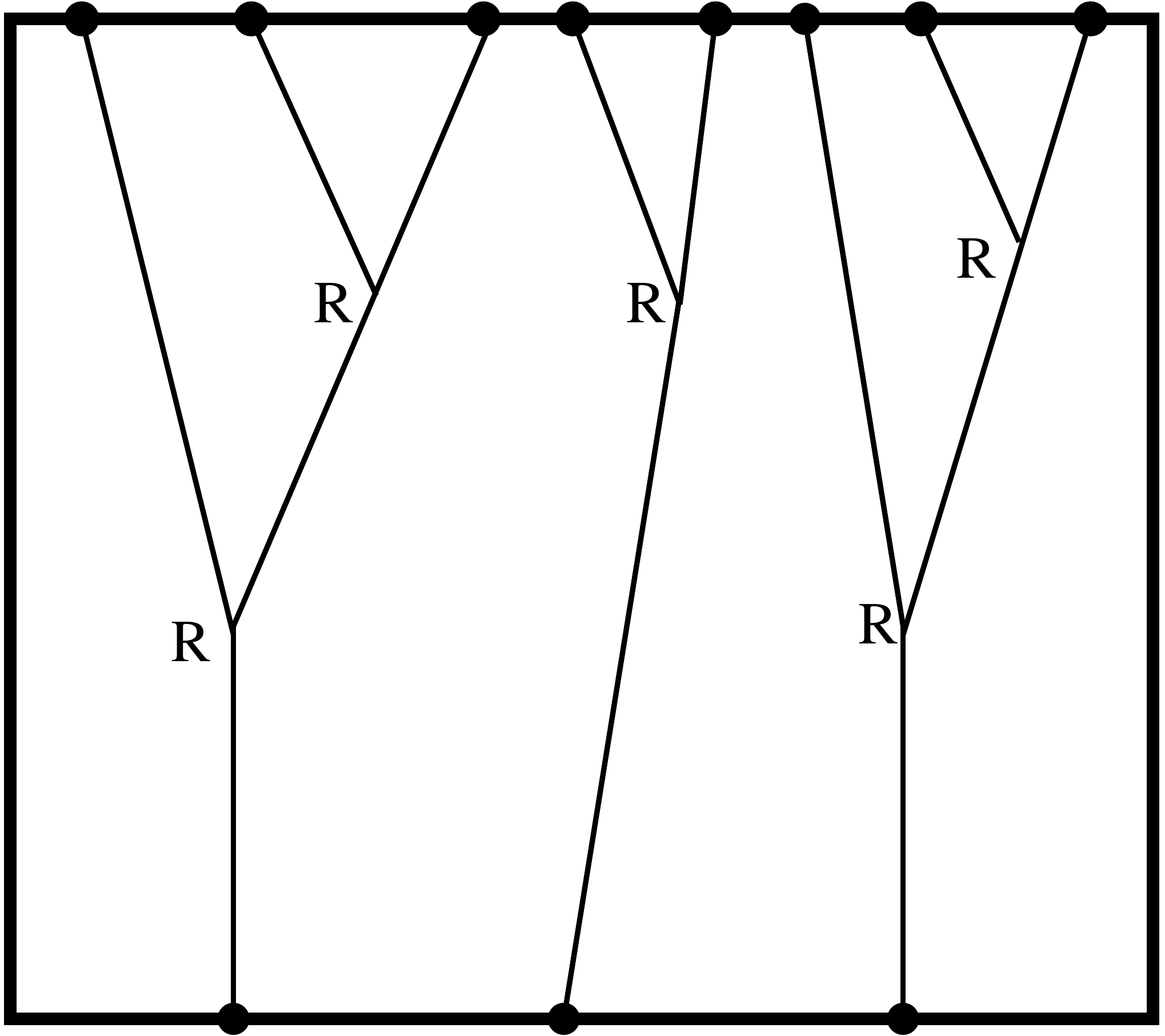} {1.5in}
 
 \end{definition}
 
 \begin{proposition} The map $\Phi_R$ defines a functor from $\mathfrak F$ to the category $Vect$ of vector spaces
 and linear maps.
 \end{proposition}
 \begin{proof} This is trivial, $\Phi_R$ takes the object $n\in \mathfrak F$ to $V_n$ and the functor property follows from
 stacking of diagrams in planar algebras.
 \end{proof}
 (If one prefers, one could take an appropriate tensor category $\mathfrak C$  with a fixed object $V$, let $V_n=\otimes^n V$ 
 and choose an element $R\in Hom(V,V\otimes V)$. The previous proposition is then just the well-known pictorial
 composition of morphisms in a tensor category.)
 
 We now come to the main object of interest in this paper.
 
 \begin{definition}\label{semicontinuous}
 Let $\mathcal P$, $\mathcal V$ and $R\in P_3$ as above be given and form the functor $\Phi_R$. By section 2 we then have
 a direct system $f\mapsto A(\Phi_R)_f$ of vector spaces on the directed set $\mathfrak T$ of binary planar rooted trees.
 The vector space $$\mathfrak V_R=\underset {\rightarrow}\lim A({\Phi_R})_f$$ will be called the \emph{semicontinous limit}
 vector space for $R$. $\mathfrak V$ contains all of the spaces $V_n$ (of $\mathcal V$) embedded one in the other by
 the maps $\Phi(\mathcal F)$ defined above.
 
 If the planar algebra has positivity, e.g. a subfactor planar algebra, and $R$ satisfies unitarity ,the inclusion maps in the direct limit are isometries so the semicontinuous limit $\mathfrak V$ has a preHilbert space structure . The Hilbert space completion of
 $\mathfrak T_R$ will be denoted $\mathfrak H_R$ and called the \emph{semicontinous limit Hilbert space}. 
 
 \end{definition}

 Note that the pre-Hilbert space structure on $\mathfrak V_R$  is preserved by  the (linear) action of $F$.
 Thus this action extends to a unitary representation $\pi_R$ on $\mathfrak H_R$.
 
 As an exercise, let us calculate a coefficient of $\pi_R$. Suppose $V_n=P_n$  and that
 $dim P_1=1$. Choose  a unit vector $\Omega\in P_0$ (the ``vacuum'').  Let $g\in F$ be given by the pair of trees 
 $(T_1,T_2)$. We want to calculate $\langle \pi_R(g)\Omega, \Omega \rangle$. Let $\iota$ be the tree with one vertex and no edges.
 By definition $\pi_R(g)\Omega$ is $ (T_1,T_2)((\iota, \Omega)).$
 See proposition \ref{action} from which we see that this can also be written $(T_1,\Phi(T_2)(\Omega)).$
 But also $\Omega=(T_1,\Phi(T_1)(\Omega)$. We see that 
$$\langle \pi_R(g)\Omega, \Omega \rangle=\langle \Phi(T_2)(\Omega),\Phi(T_1)(\Omega)\rangle,$$
the inner product being taken in the planar algebra. For instance if \\
$T_1=$ \vpic {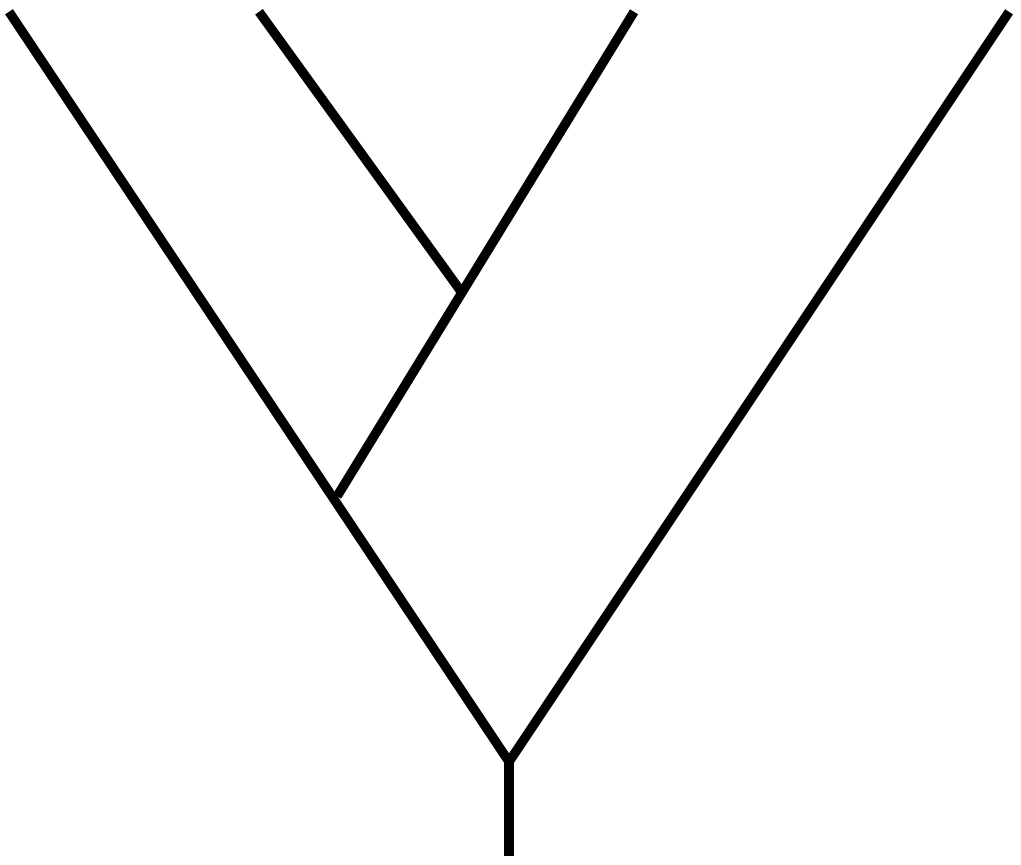} {1in} and $T_2=$ \vpic {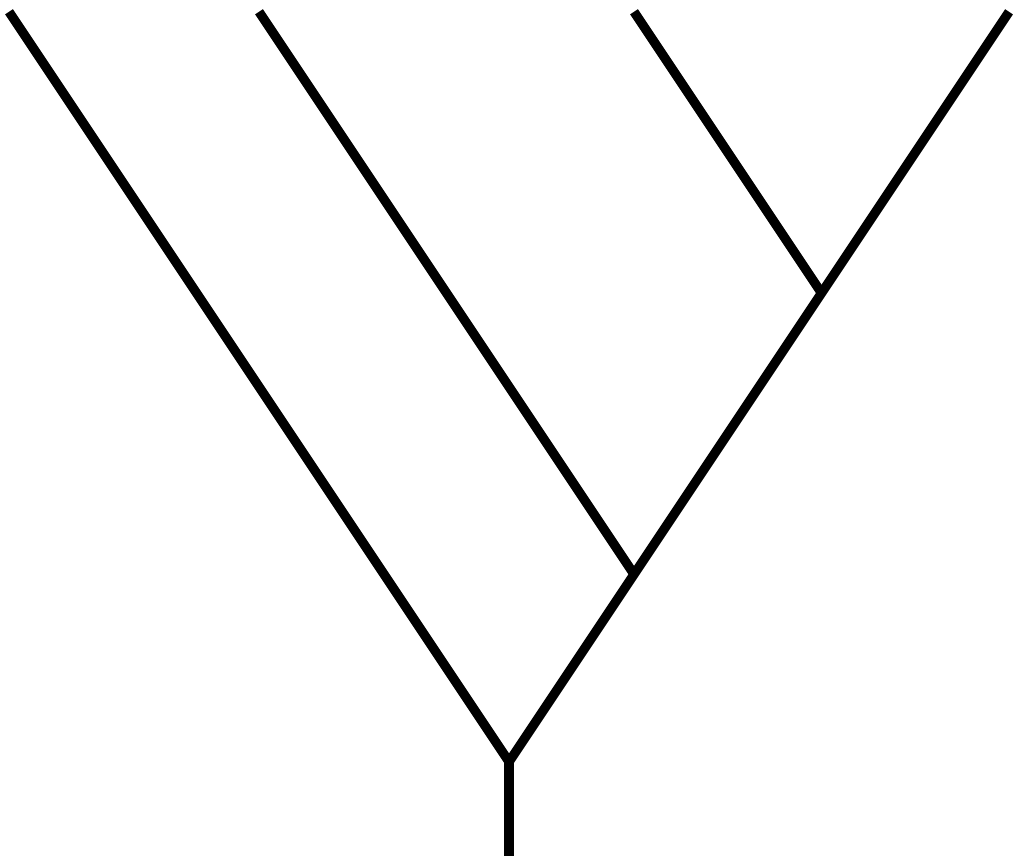} {1in} 
then \\
\vpic {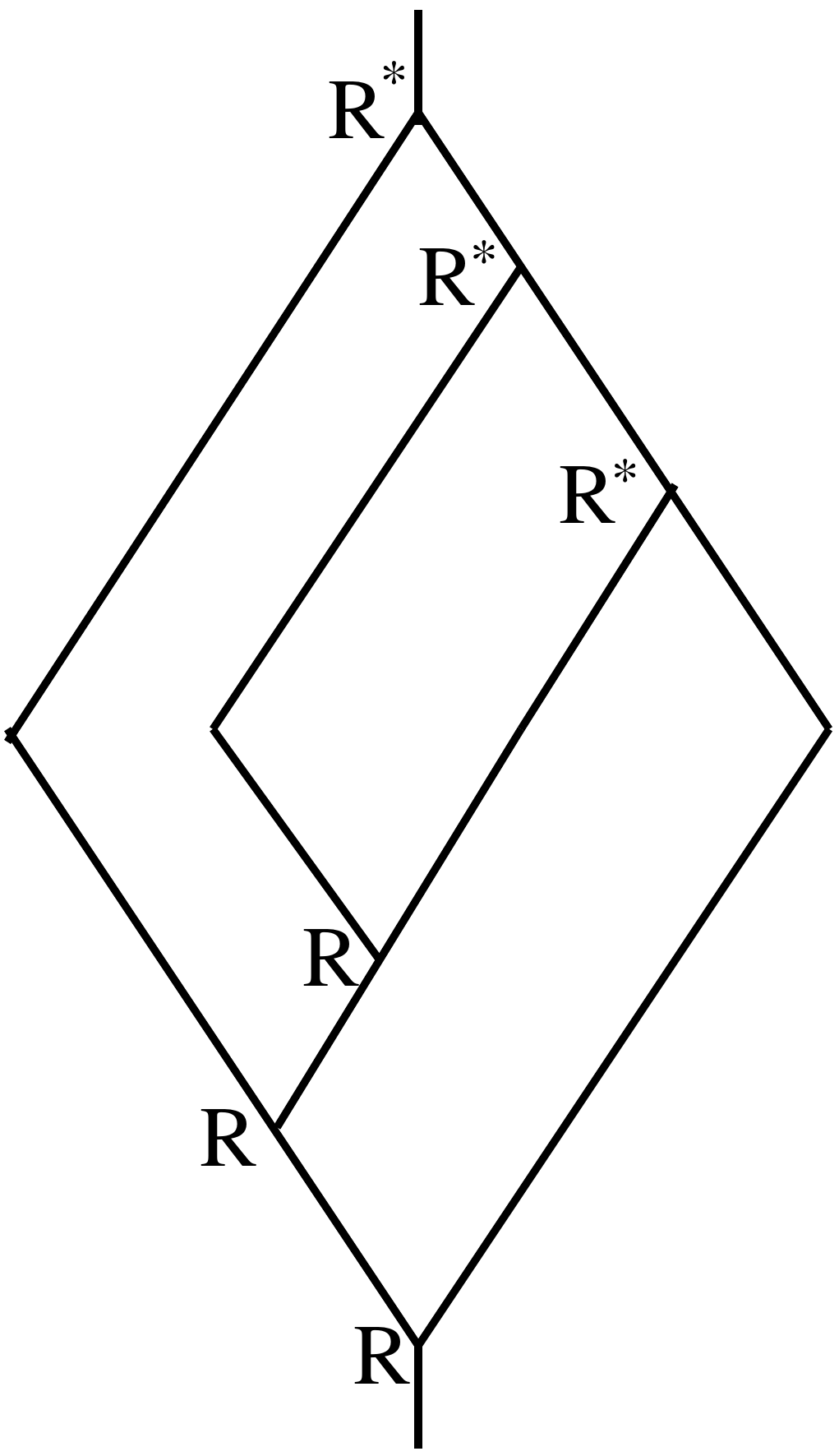} {1in}  = $\displaystyle \langle \pi_R(g)\Omega, \Omega \rangle$  \vpic {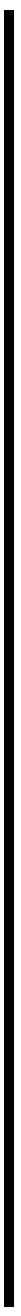} {0.0273in} \\
so that if $\delta$ is the loop parameter of the planar algebra then\\
\vskip 5pt
$\displaystyle \langle \pi_R(g)\Omega, \Omega \rangle=\frac{1}{\delta}$ \vpic {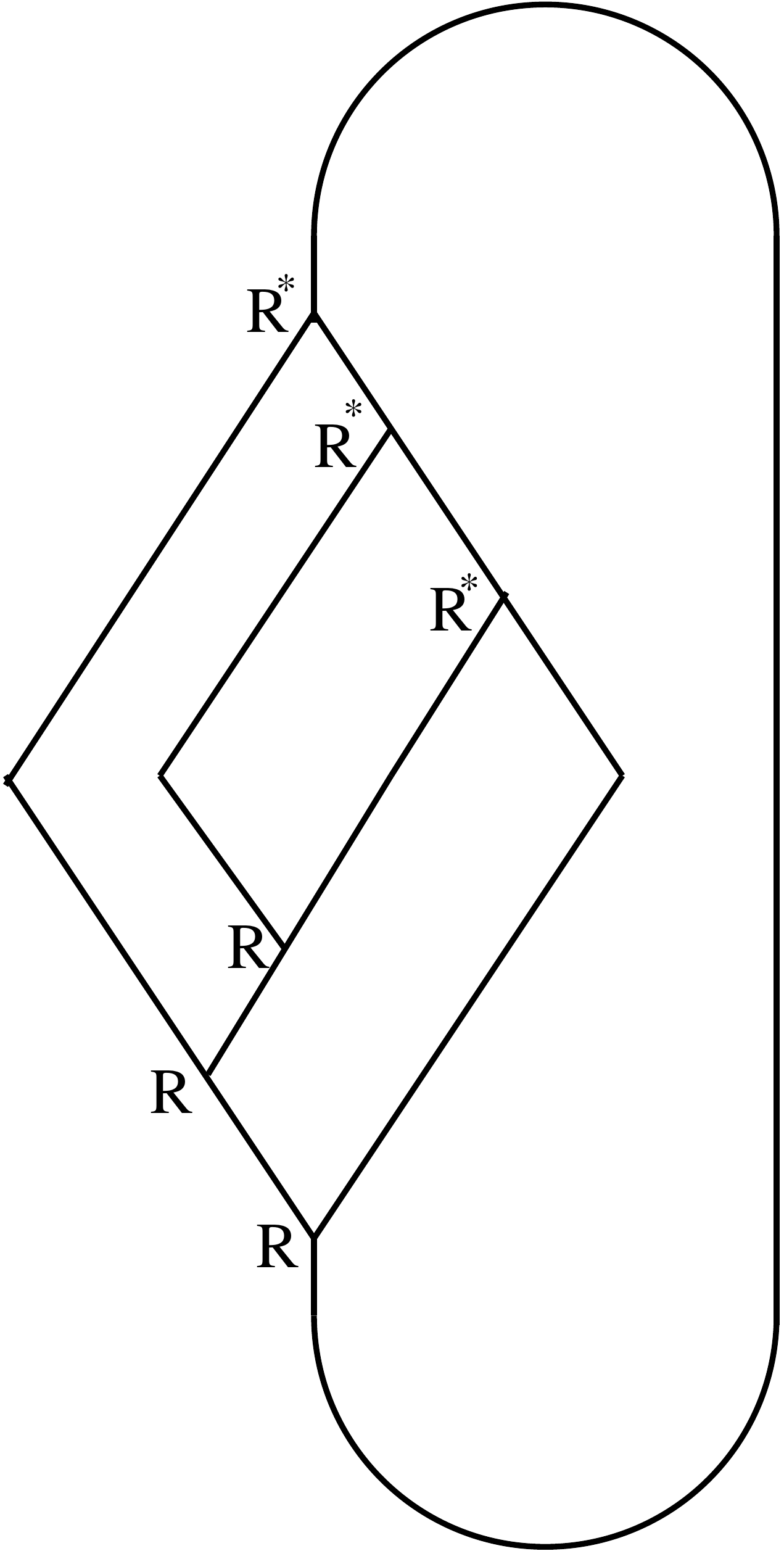} {1in}

There are many interesting choices of $R$. For instance if the planar algebra is the tensor planar algebra 
on a vector space of dimension $3$ with orthogonal basis $\{1,2,3\}$ then we may define $R$ to be the
$3$-tensor\\
$R_{i,j,k} = \begin{cases} 0 &\mbox{if } i=j \mbox{ or } j=k \mbox{ or }  i=k \\ 
1 &\mbox{otherwise } .
\end{cases}
$
 Then $\displaystyle \langle \pi_R(g)\Omega, \Omega \rangle$ is equal to the number of ways of 3-colouring the 
 edges of the three valent graph underlying the diagram for this inner product in such a way that the 3 colours at
 any vertex are distinct. The positivity of these coefficients for all $g\in F$ is known to be equivalent to the 4-colour theorem!
 We are grateful to Roland Bacher for pointing this out-see \cite{rt}.
 
 Or, if the planar algebra is the version of the tensor planar algebra in which the $n$ indices sit in the regions and\\
 
\qquad \vpic {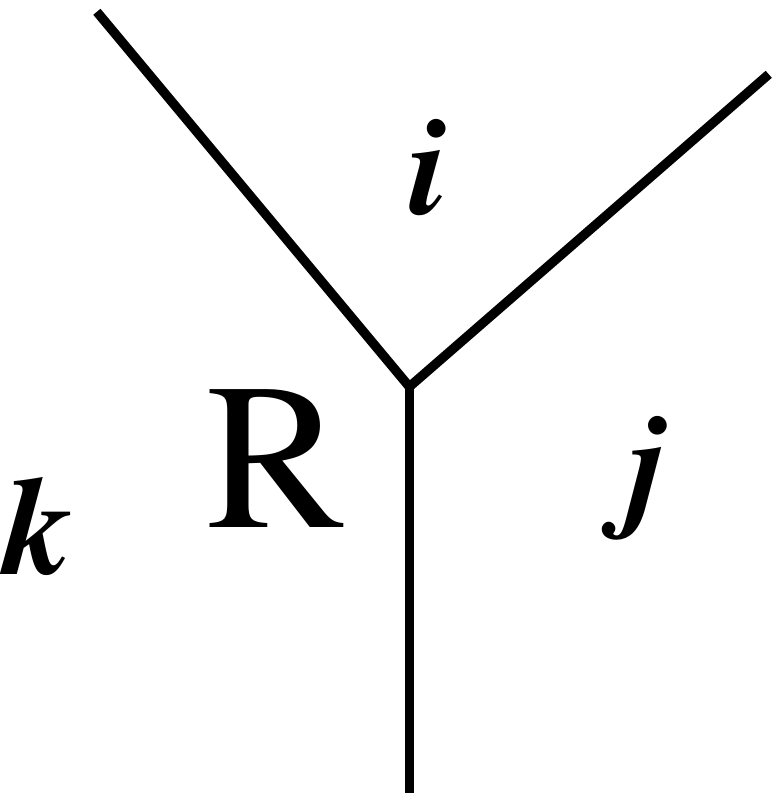} {1in} $= \begin{cases} 0 &\mbox{if } i=j \mbox{ or } j=k \mbox{ or }  i=k \\ 
1 &\mbox{otherwise } .
\end{cases} $\\
then $n\displaystyle \langle \pi_R(g)\Omega, \Omega \rangle$ is the number of ways of $n$-colouring the map defined
by the diagram for this inner product.  If $n=3$ the map can  be coloured in 6 ways or not at all so we find
that the set of all $g$ for which $\displaystyle \langle \pi_R(g)\Omega, \Omega \rangle=2$ is a subgroup of $F$.
Yunxiang Ren has shown that it is isomorphic to the Thompson group $F_4$.

\item{Thompson's group T.} Here things work almost exactly as they do for $F$. One takes the same kind of $R$ as before and
 an affine representation of the planar algebra. Replacing the vertices of  morphisms in $\mathfrak{AF}$ by discs containing
 $R$ gives a functor from $\mathfrak{AF}$ to Vect, taking $n$ to the $n$ vector space of the affine representation. We thus
 get a representation of $\tilde T$. If the representation is irreducible and the rotation acts as a scalar then one obtains a projective
 representation of $T$. The unitary affine representations of the Temperley-Lieb algebra are well understood for all values of
 the loop parameter - see \cite{gl},\cite{jo3} and \cite{JR}. The same examples of $R$ as for $F$ yield similar interpretations of coefficients.
 
 \begin{definition}\label{annularsemicontinuous} The semicontinuous limit vector space and Hilbert space $\mathfrak V_R$ 
 and $\mathfrak H_R$ are defined in exactly
 the same way in this annular context they were in \ref{semicontinuous} for rectangular representations of the planar algebra.
 \end{definition}
 
 Note that $\tilde T$ acts unitarily on $\mathfrak H_R$ if the planar algebra has positivity and $R$ satisfies unitarity.
  
  It is clear that the projective representation of $T$ will be an ordinary representation if the affine representation is
  in fact annular (see \ref{annular}).
  
 \item{Thompson's group V.} The representations are easiest to describe if we use the tensor planar algebra based on
 an underlying vector space $V$.  We can choose any 
 tensor $R$ with three indices satisfying the unitarity condition. A permutation $\pi$  in $S_n$ defines a linear map
 $\otimes^n \pi$ on $\otimes^n V$ by permuting 
 coordinates so for $(\mathcal F,\pi) \in Mor_{\mathfrak V}(m,n)$ we may define $\Phi((\mathcal F, \pi))$   by filling in $\mathcal F$'s vertices with discs
 containing $R$ as before then composing the corresponding linear map from $\otimes^m V$ to $\otimes^n V$ with $\otimes^n  \pi$. 
 It is easy to check that this $\Phi$ is a functor and hence defines a unitary representation of $V$.
 
 Note that there is a purely diagrammatic way to represent the category $\mathfrak V$ by drawing permutations as strings connecting
 permuted points. So if one could find a planar algebra $(P_n)$ with an element of $P_4$ satisfying the obvious relations of a transposition:\\
 
  \vpic {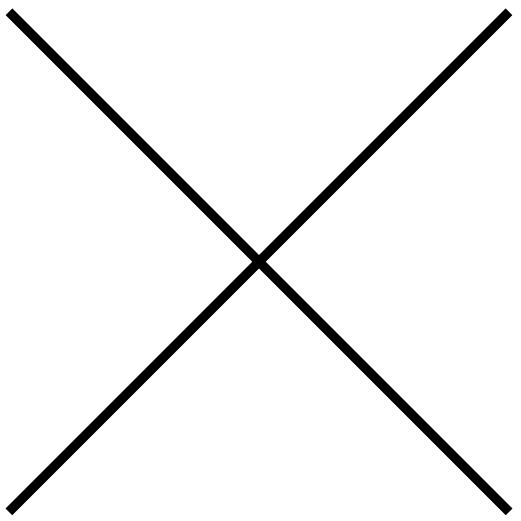} {0.5in} such that  \vpic {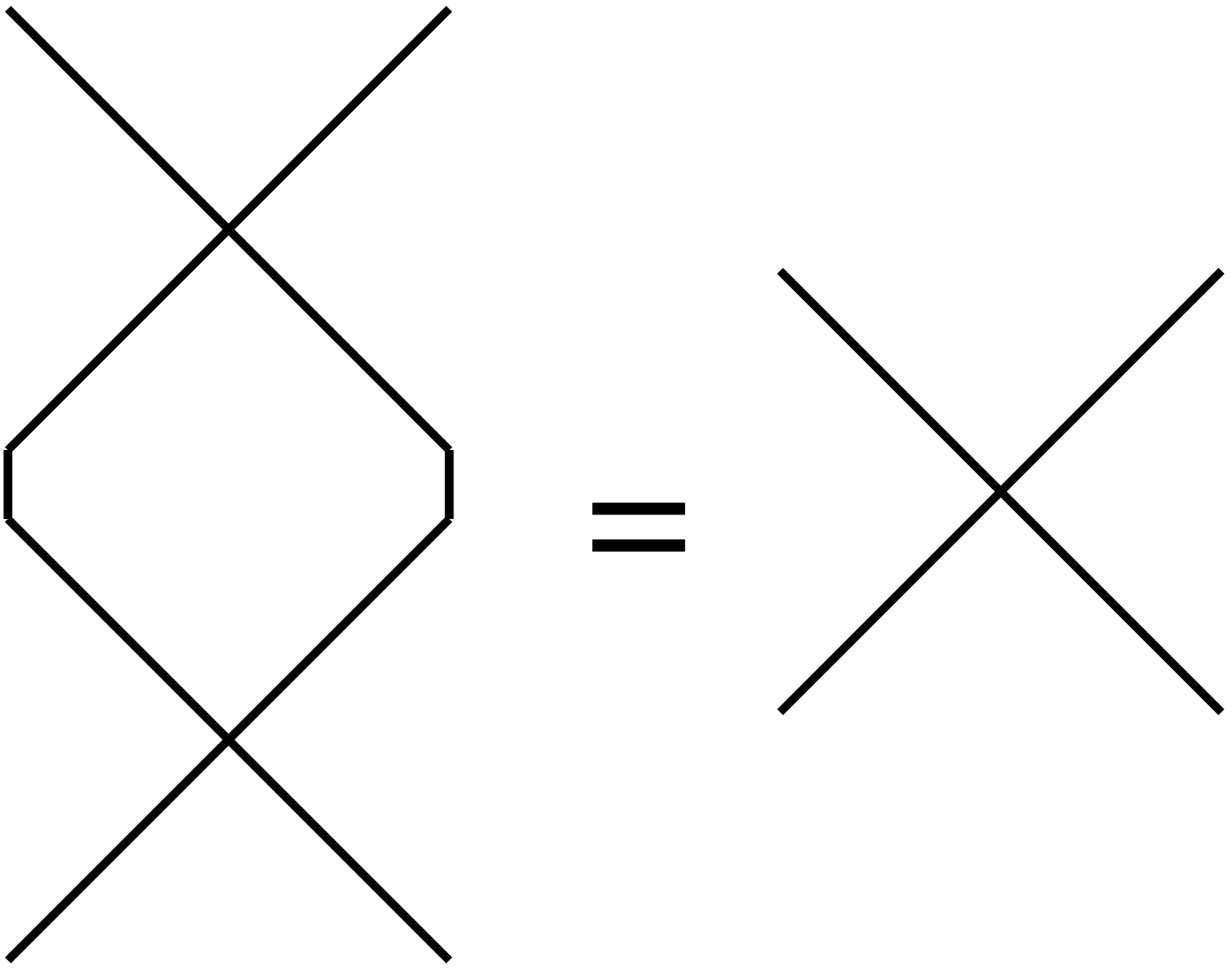} {1in} and \vpic {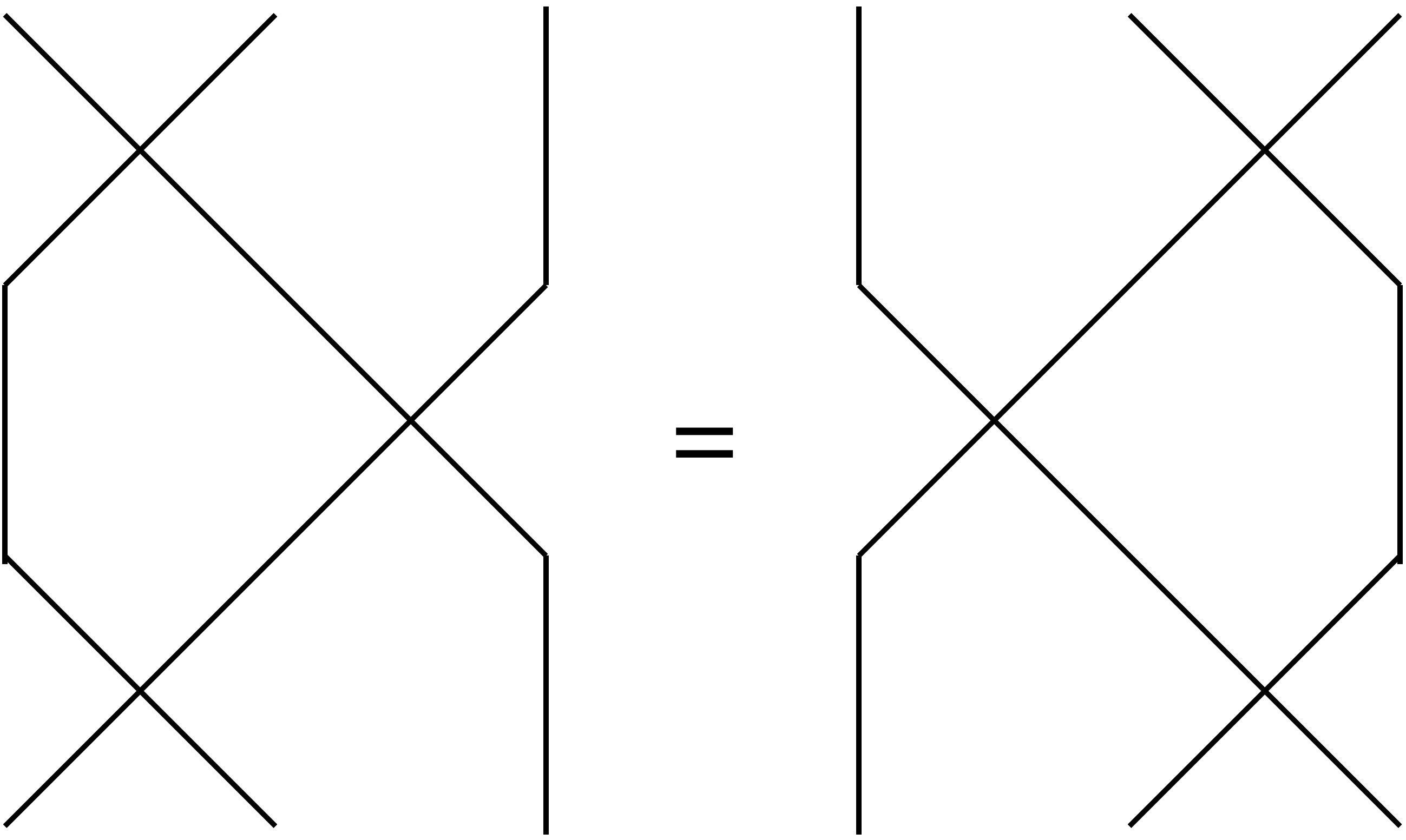} {2in}
 
\noindent then one will get a representation of $V$ provide the $R\in P_3$ and the crossing in $P_4$ satisfy:\\
 
\hspace{1.5in} \vpic {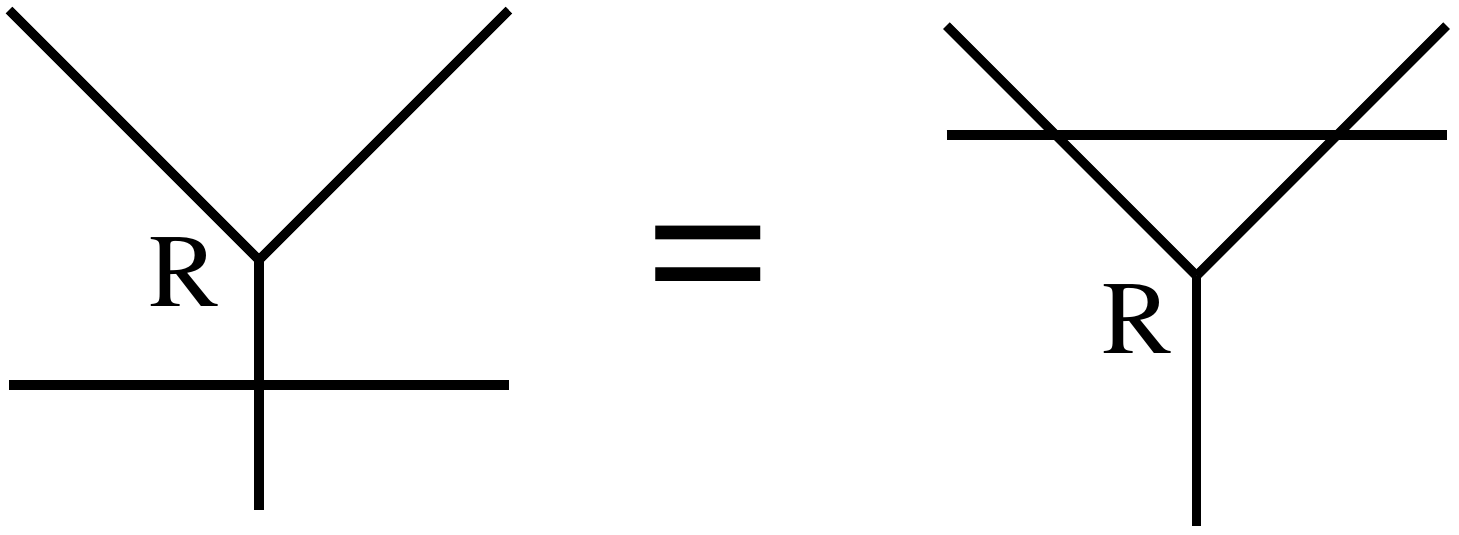} {1.8in}
 
 This condition is automatic if the crossing is the transposition acting on $V\otimes V$.

\item{Braided Thompson group.} This works just like for $V$.  There is a purely diagrammatic representation of 
morphisms in $\mathfrak {BF}$ which is just like the the one for $\mathfrak V$ except that the transposition is allowed to be 
a crossing:\\
 
 \hspace {2in} \vpic {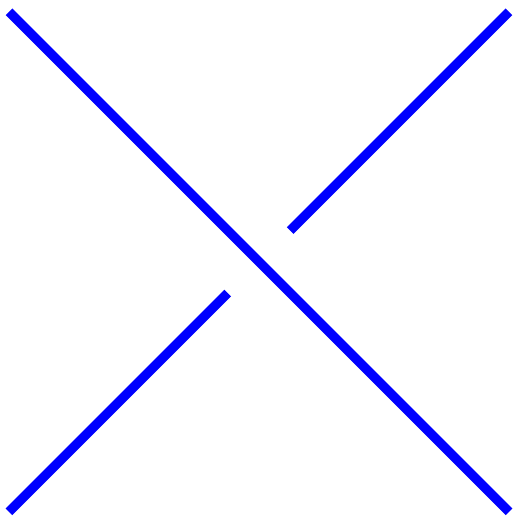} {0.7in}
 
 satisfying the braid group relations and the following two relations with the vertices of the trees:\\
 
 \vpic {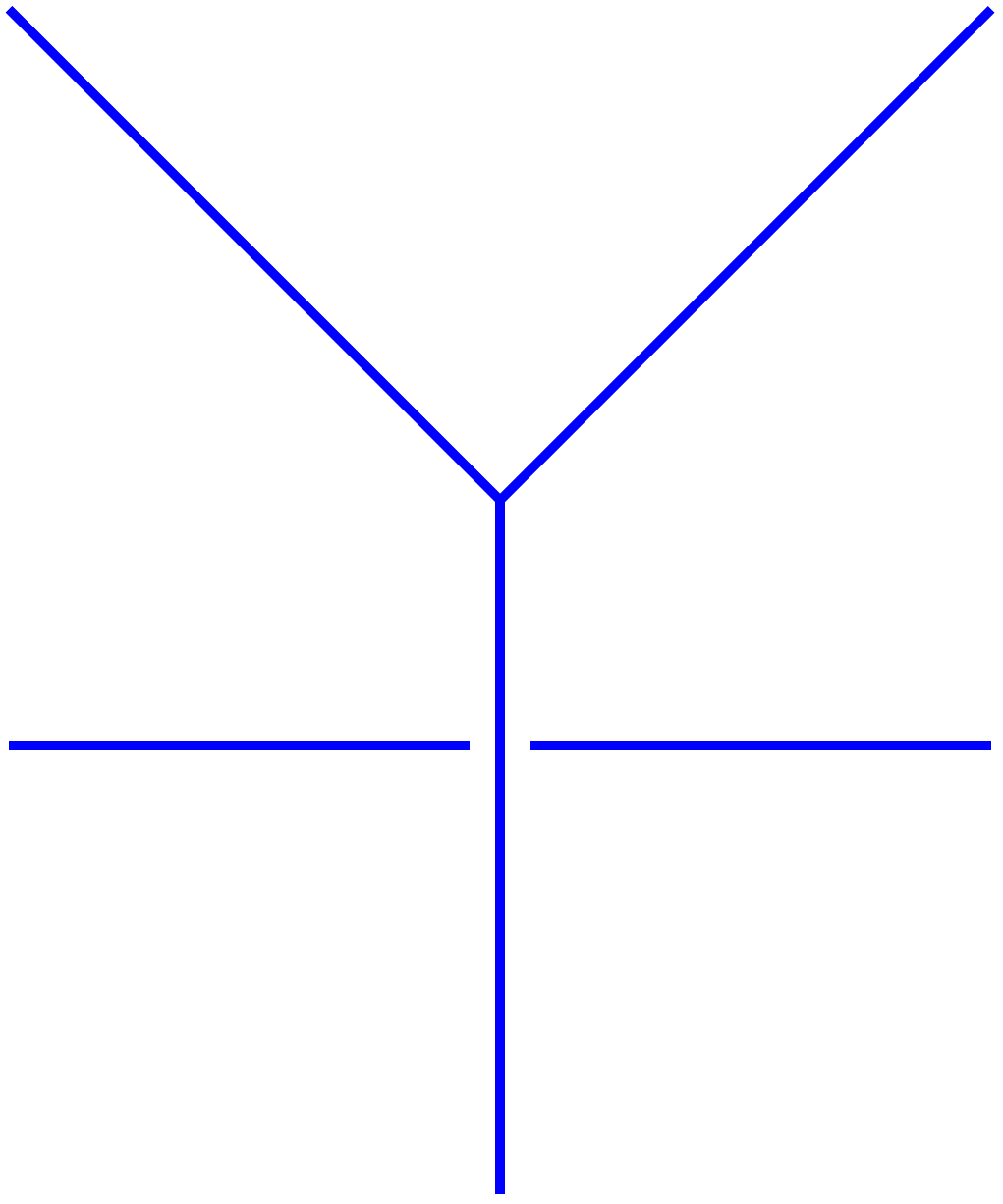} {0.7in}  $=$ \vpic {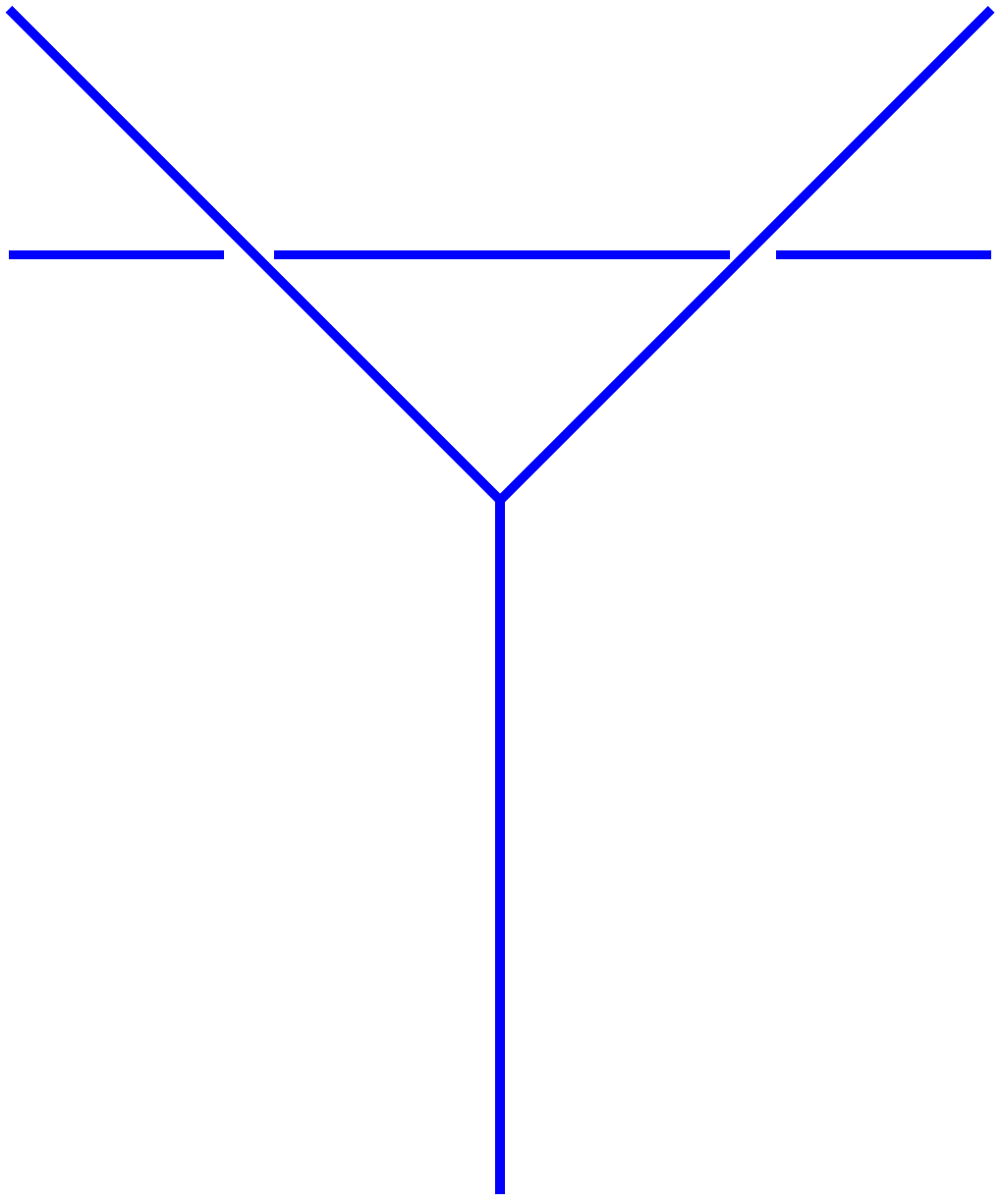} {0.7in}  and \vpic{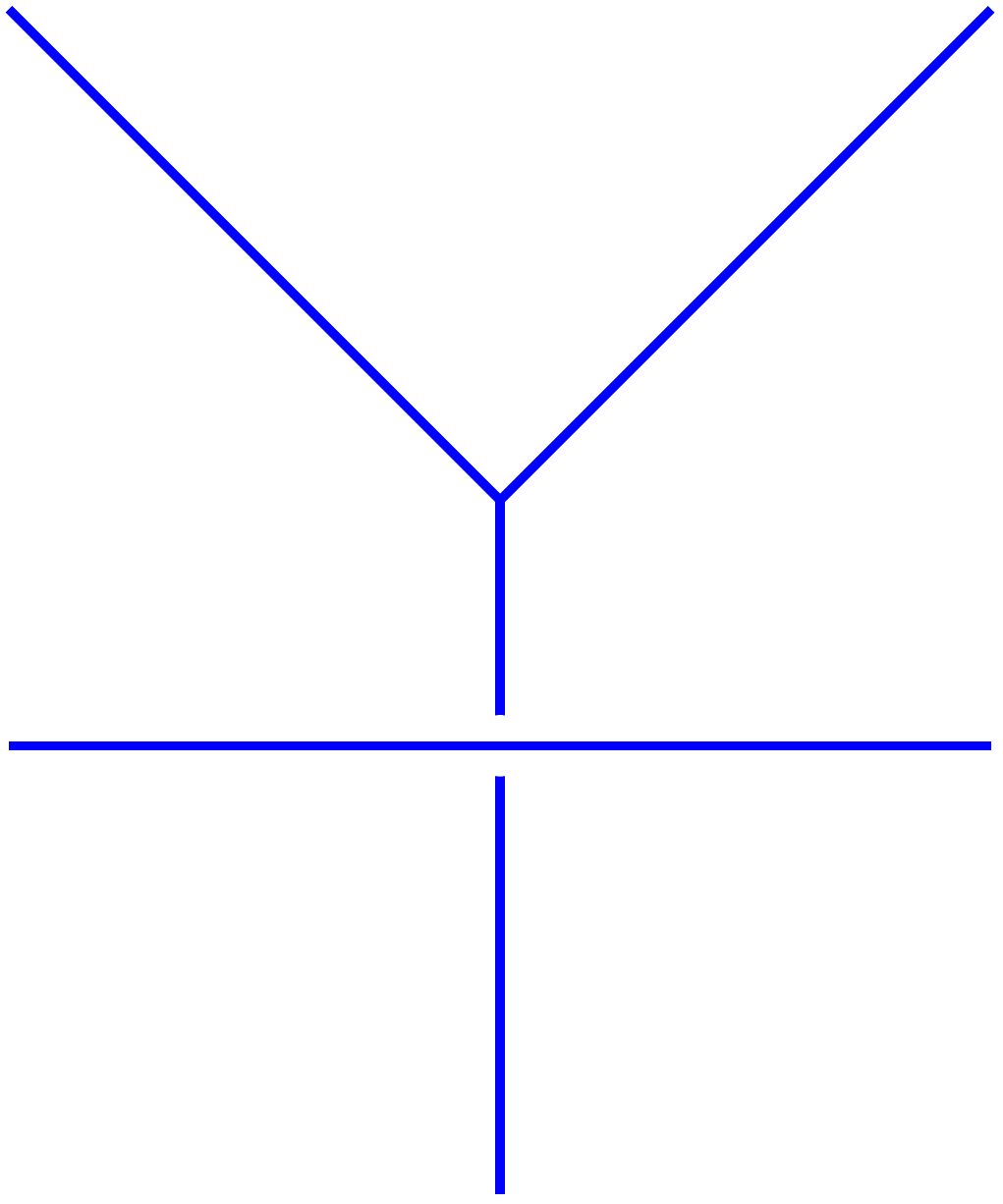} {0.7 in}  $=$ \vpic {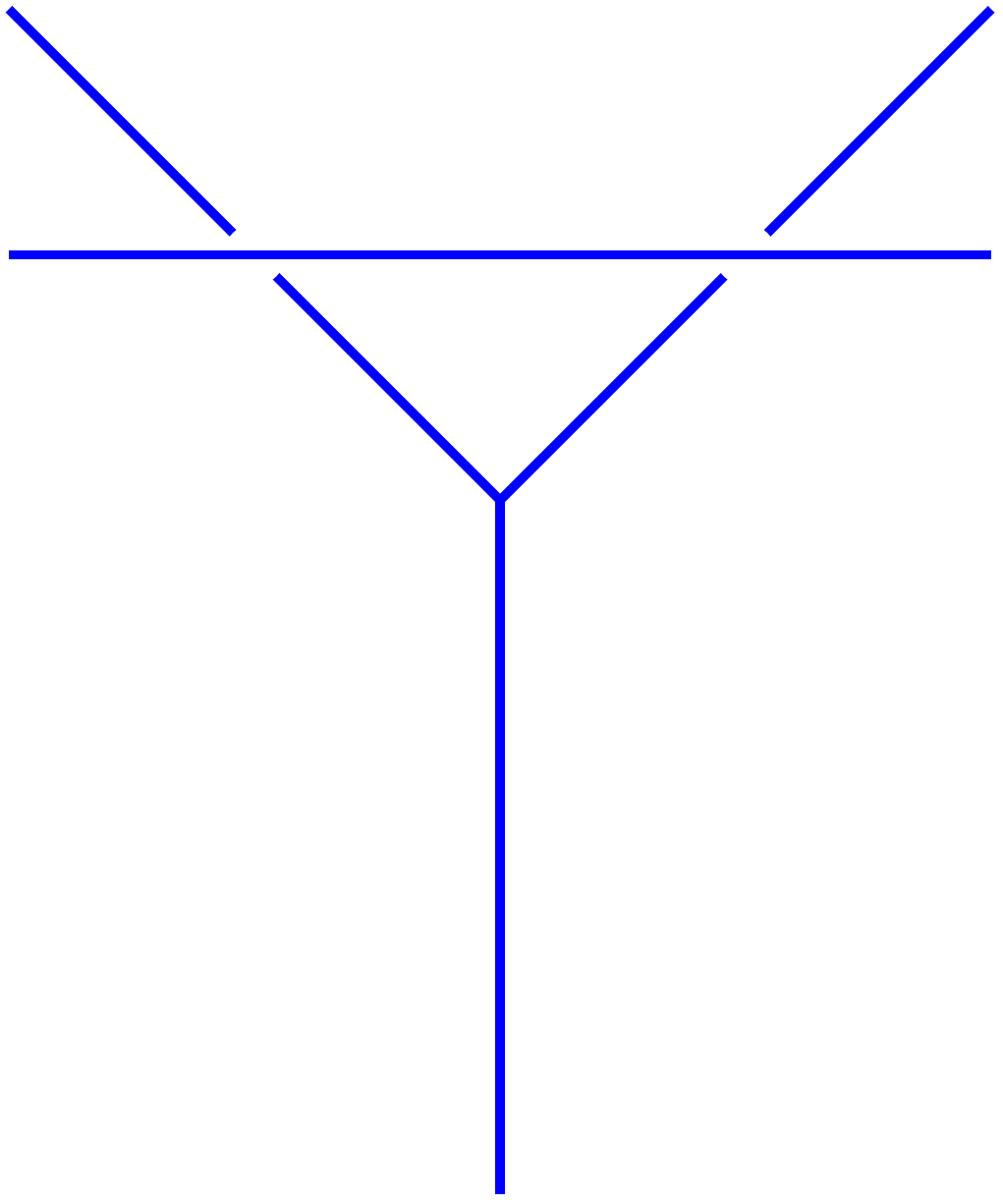} {0.7 in}  .
 
 Representations of $\mathfrak{BF}$ are easy to come by in planar algebras/tensor categories. Coeffiecients
 of the form $\langle g\Omega,\Omega\rangle$ are just the partition functions in the corresponding planar algebra.
 
 Let us make one curious remark. We saw that the braid group can be obtained as the group of fractions of the
 semigroup of positive braids. So it is with $\mathfrak {BF}$ which we can make smaller by requiring that all 
 the crossings be positive. Then to obtain a representation of $BF$ we only need the first of the two relations
 above between the crossing and the trivalent vertex. We have not investigated this.
 
\end{enumerate}

\subsection{The relation between these constructions and those of \cite{jo4}.}
The representations of $F$ obtained in \cite{jo4} may be obtained by the construction of this paper by first 
embedding $F=F_2$ in $F_3$ by taking a pair of trees $(T_1,T_2) \in F_2$ and adding strings to turn all the trivalent
vertices into quadrivalent ones, obtaining the pair $(\hat T_1,\hat T_2)$  as illustrated below:\\

$(T_1,T_2)= \big {(}$ \vpic {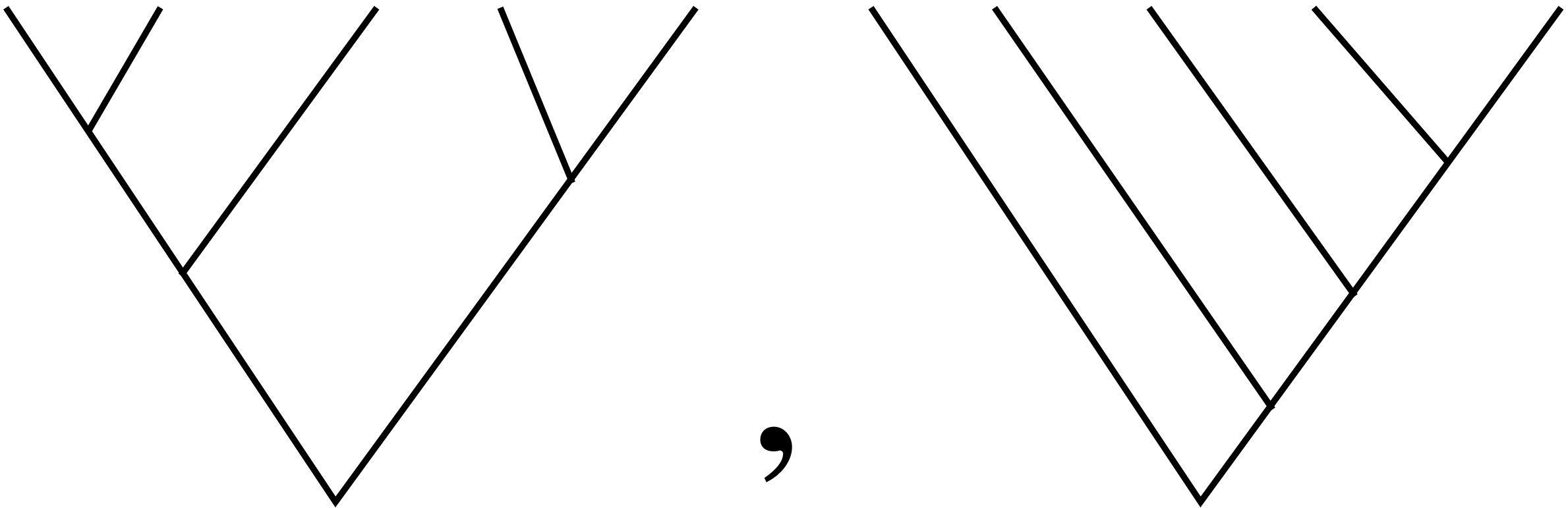} {1.3in} $\big )\qquad \rightarrow (\hat T_1, \hat T_2)= \big ($ \vpic {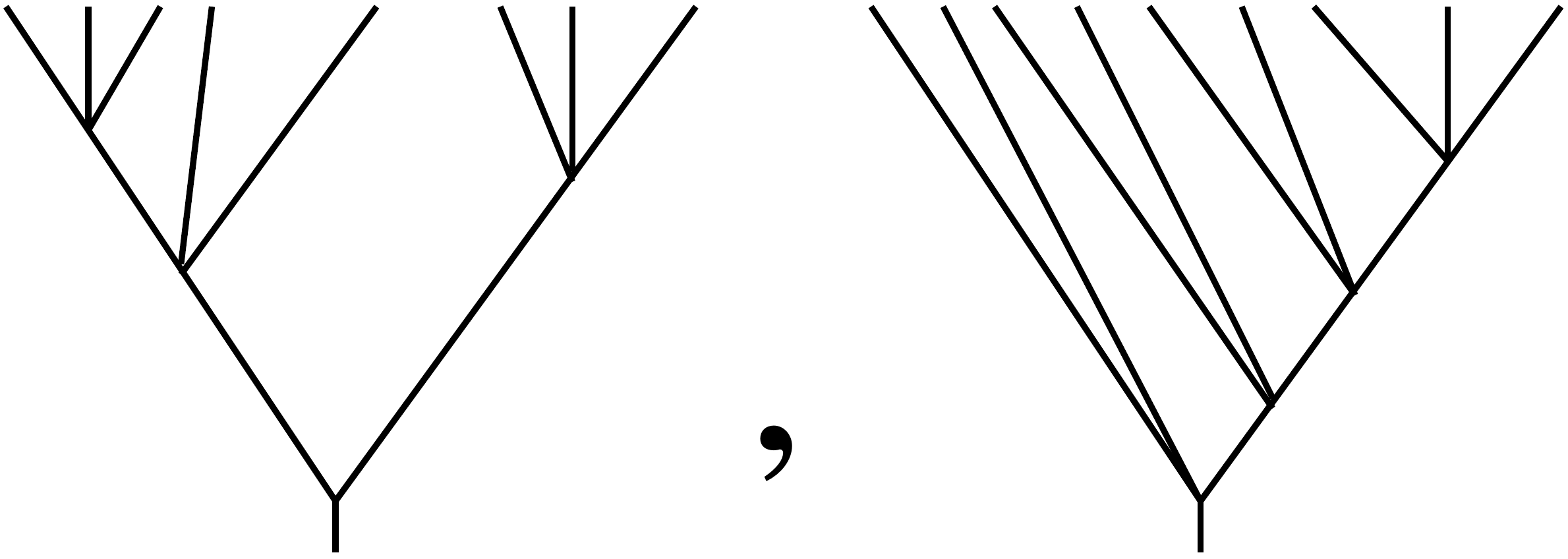} {1.3in} $\big )$

Thus it was possible to use elements $R\in P_4$ rather than $R\in P_3$ to obtain representations of $F_2$. Otherwise the 
construction of representations of $F$ \cite{jo4} was just a more clumsy version of what we have done in this paper  in much greater
generality.

\section{The NoGo theorem.}

The Thompson group $T$ contains the subgroup $Rot$ of rotations of the circle by dyadic rationals.
We will prove that the representation of $Rot$ on the semicontinuous limit Hilbert space (from an affine
representation of a positive definite planar algebra) is highly dicontinuous if we topologise $Rot$ as 
a subgroup of $\mathbb R /\mathbb Z$. 
This is not at all surprising. The geometric structure underlying the semicontinuous limit is the full binary tree whose
branches are dangling and do not feel the topology of the circle. The discontinuity result is true in
great generality but we will only prove it for a single family of planar algebras (with positive definite inner product). We have chosen
this family because there is, up to an irrelevant scalar, only one choice of $R$.  

To be precise, let $\mathcal Q=(Q_n)$ be the planar algebra obtained from the TL planar algebra with loop parameter 
$\delta = 2\cos \pi/n$ for $n=6,7,8,9,\cdots$ by cabling
2 strings and cutting down by the JW idempotent (this is also quantum $SO(3)$ at a root of unity). See \cite{MPS}. We can represent the JW idempotent in $TL_4$
with 4 boundary points as a box \vpic {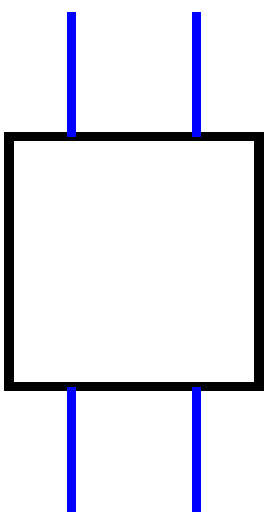} {0.3in} , entirely defined by the relations \\ \vpic {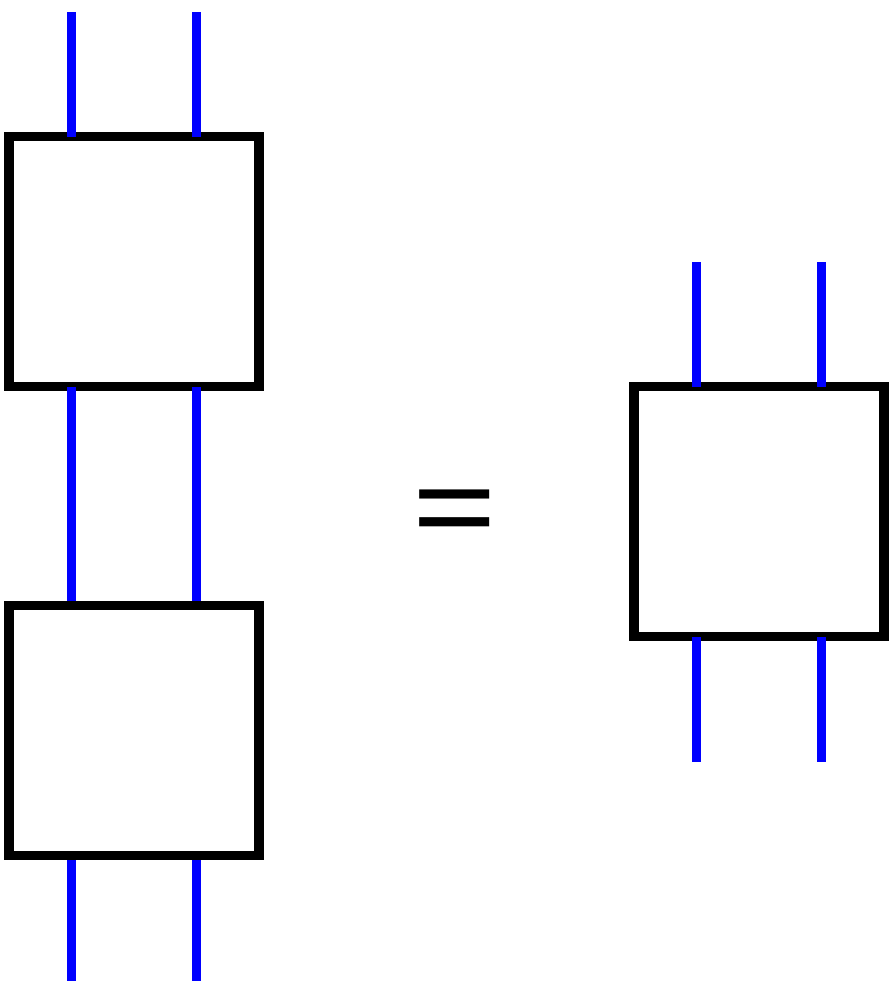} {0.8in} , \vpic {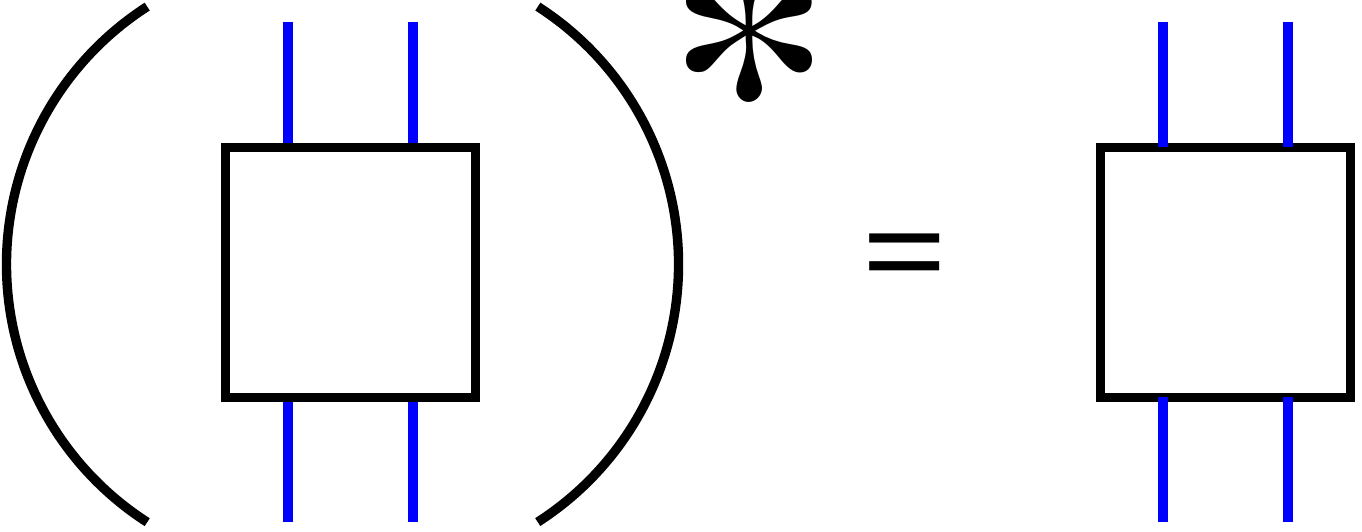} {1in} and 
\vpic {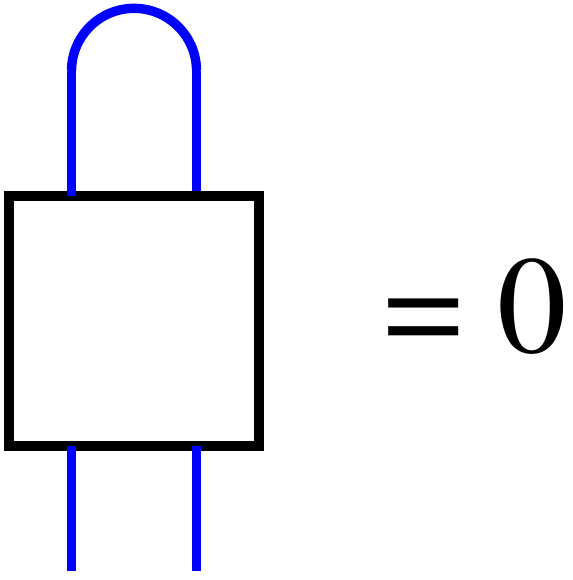} {0.6in} . 
One checks that \vpic {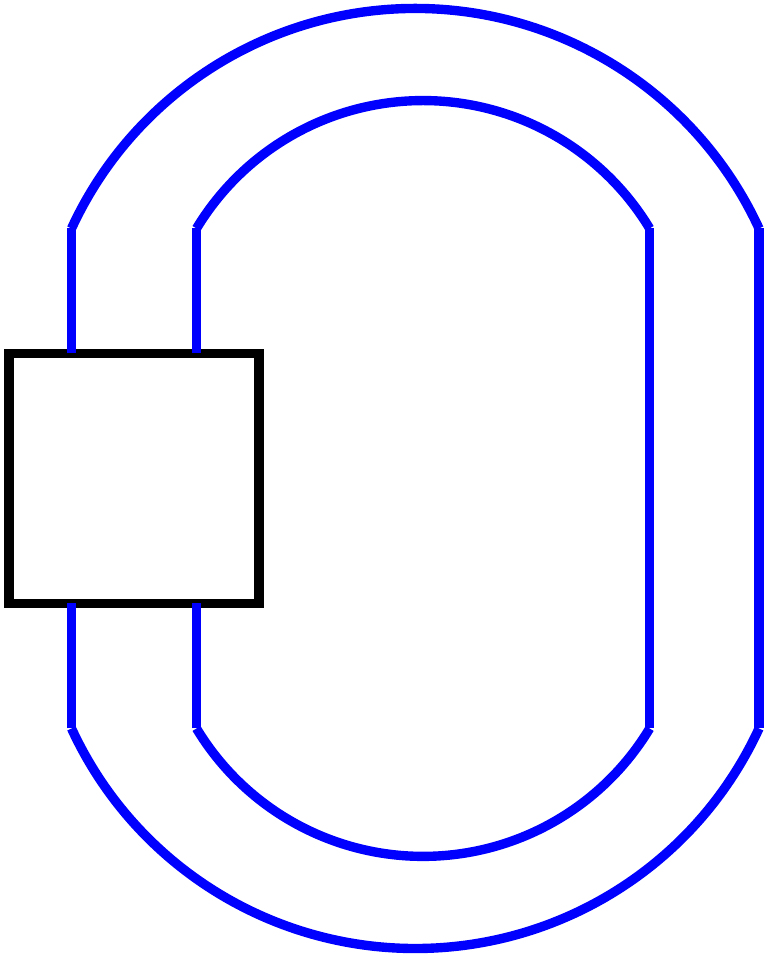} {0.6in} $= \delta^2-1$.
In $\mathcal Q$ it is well known that $Q_3$ is spanned by the single element $$R=\sqrt{\frac{\delta}{\delta^2-2}} \vpic {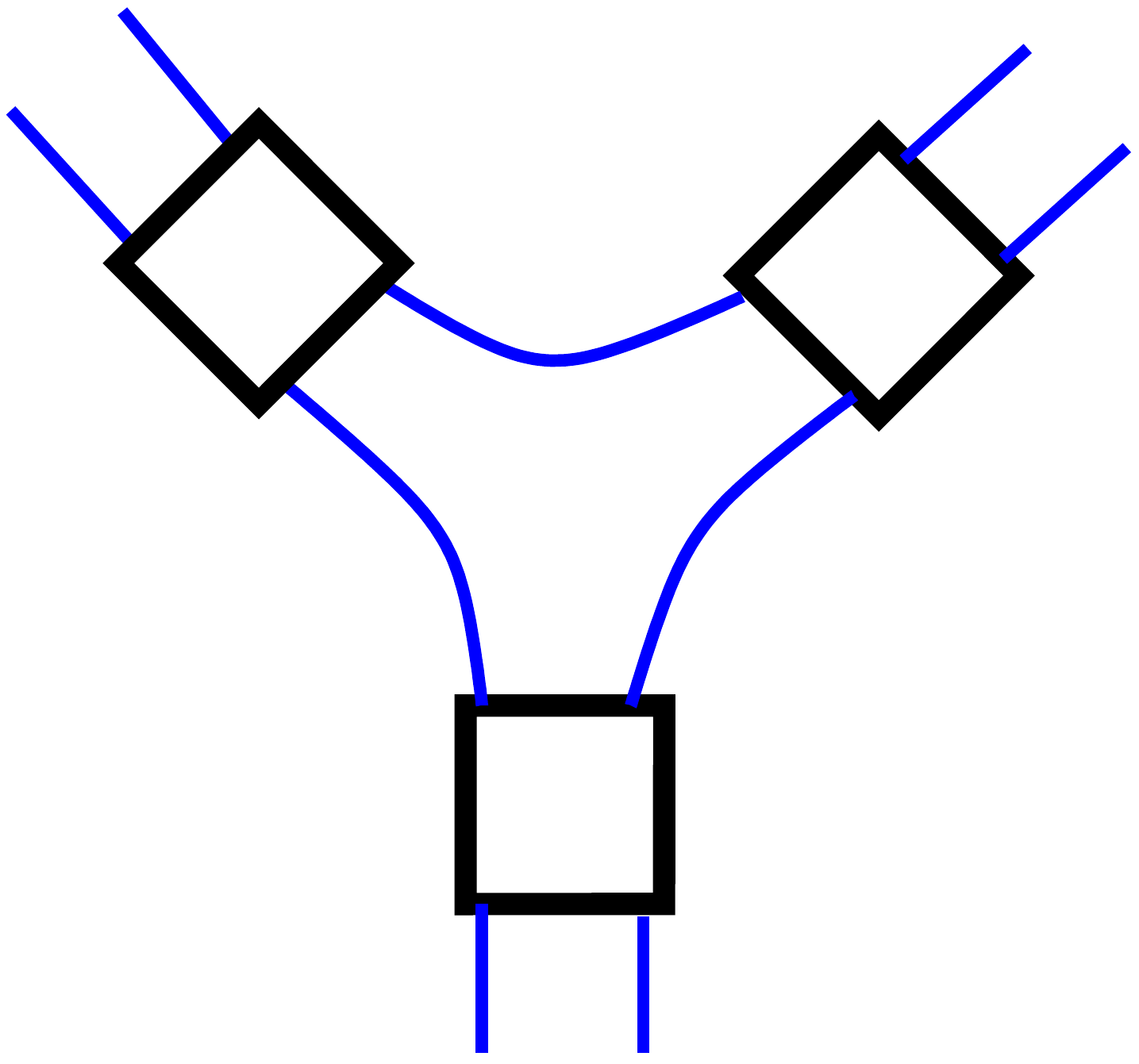} {0.5in}  ,$$
the normalisation guaranteeing unitarity. 
$\mathcal Q$ is obtained by combining the cabled strings to a single string. Thus $R$ is an element of $Q_3$ and in
$\mathcal Q$ the loop parameter is $d=\delta^2-1$. 
  
 \emph{ Note that $R$ is rotationally invariant so we will  suppress it in all pictures, i.e. from now on} \vpic {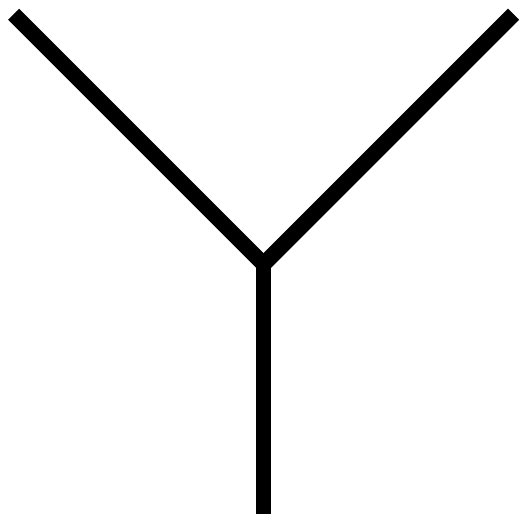} {0.3in}   
 \emph{ will mean } \vpic {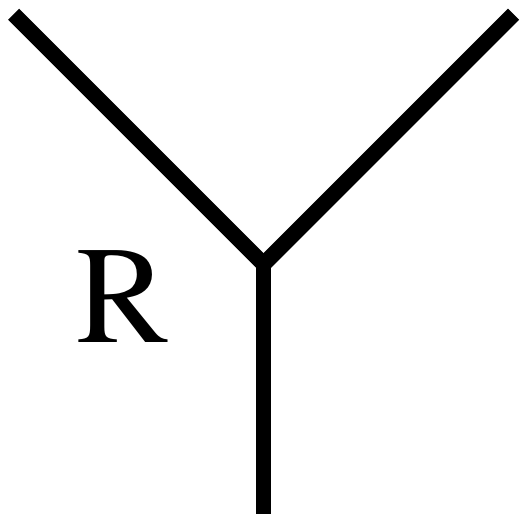} {0.3in}  .

  Now let $\mathcal H$ be the semicontinuous limit Hilbert space constructed in \ref{semicontinuous} and \ref{annularsemicontinuous}
from the planar algebra $\mathcal Q$ and an annular representation $\mathcal V=(V_n)$ of it, using the element $R$ defined above to construct the functor $\Phi$.
By section \ref{construction} we know that Thompson's group $T$ acts unitarily on $\mathcal H$. In particular for every dyadic rational 
$r=\frac{a}{2^n}\in [0,1)$ we have a unitary $\rho_r$ on $\mathcal H$ representing the rotation by $a$ in $T$.

We have proved the following for $d=4\cos^2\pi/n +1$ for $n\in \mathbb N, 7\leq n\leq 20$ and $d=3$. It is 
false for $n=5$ and $n=6$ and surely true for all $n\geq 7$.

\begin{theorem} For any vectors $\xi,\eta\in \mathcal H$,
\quad $\displaystyle \lim_{n\rightarrow \infty}\langle \rho_{\frac{1}{2^n}}\xi,\eta\rangle=0.$
\end{theorem}
\begin{proof} Note that since the representation is unitary we may suppose that $\xi$ and $\eta$ are actually in 
some space $(T,V_{2^k})$ where $T$ is the annular tree \\ 
\mbox{                               } \vpic {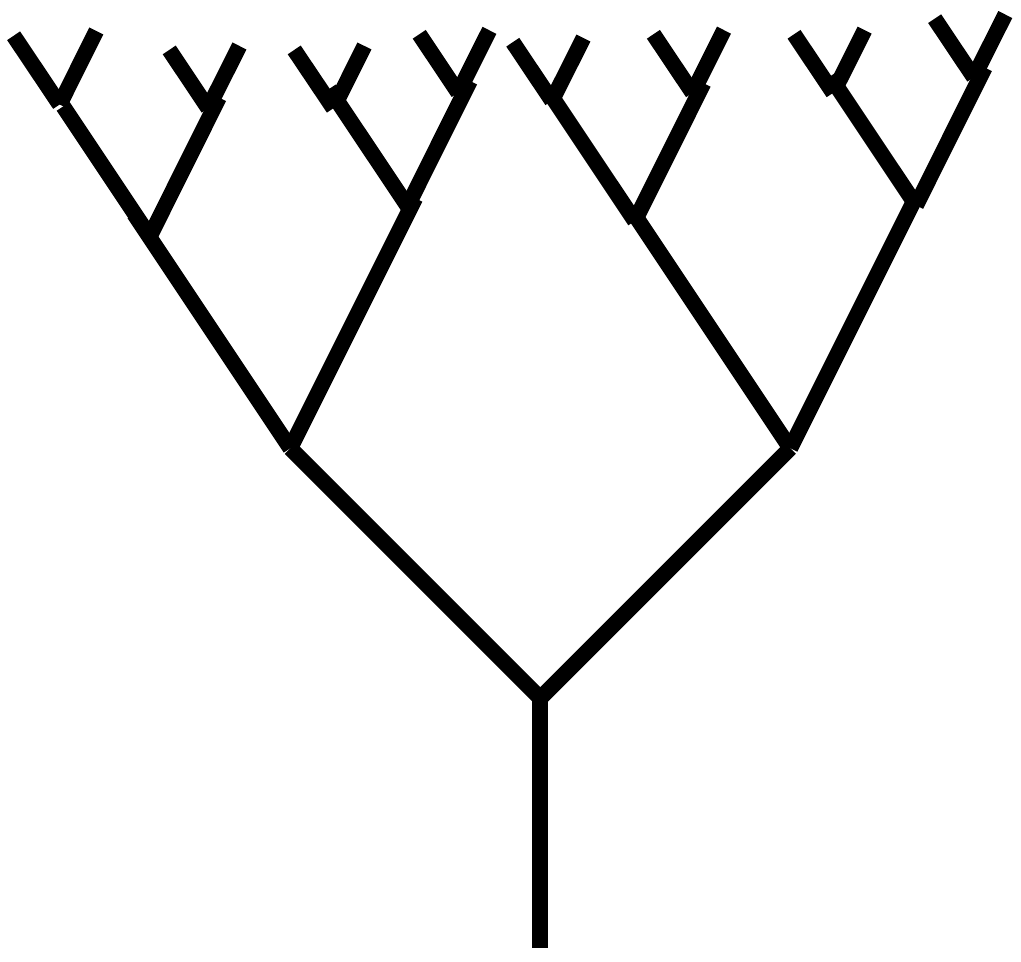} {1in} (shown for $n=4$).
The following diagram is  $\langle \rho_{\frac{1}{2^{k+n+1}}}\xi,\eta\rangle$ which we illustrate
here for $k=1$ and $n=3$. Note that we are applying periodic boundary conditions.  

\vskip 5pt

\vpic {gxieta} {3.5in}

Now all the regions in the blue dotted circles can be isotoped to look like 

\vpic {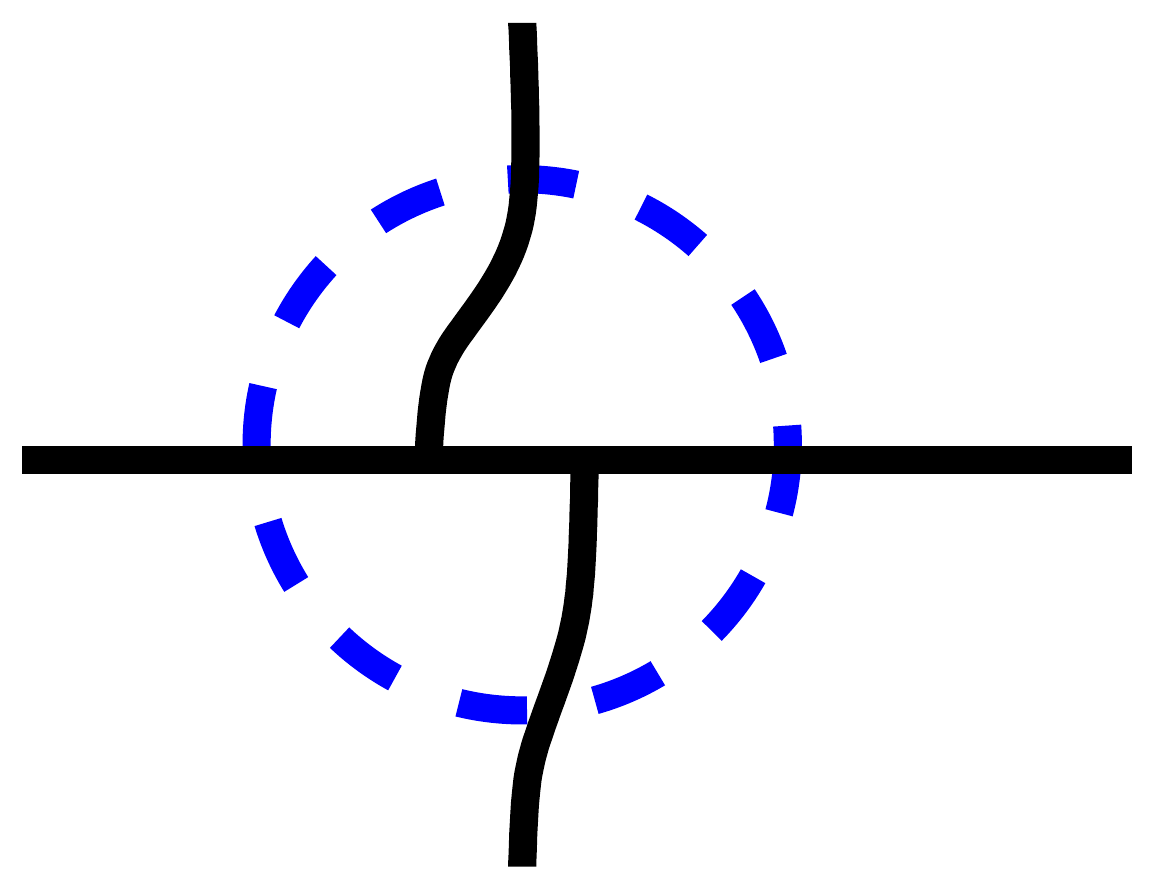} {1in} so if we call  $x$ this element of $Q_4$ the picture becomes:

\vpic {transferforrot} {3.5in}

(The positions of the \$ signs in the picture are obvious, so suppressed.)

We recognise  the {\em transfer matrix} (see appendix \ref{affinecat}) $T_{2^{n+k}}(x)$ ! 

\begin{definition}\label{renormalgebra}  We define the bilinear map
$\mathscr{B}:Q_4\times Q_4 \rightarrow Q_4$ by\\

\hspace {2in} $\displaystyle \mathscr{B}(x,y)=$ \vpic{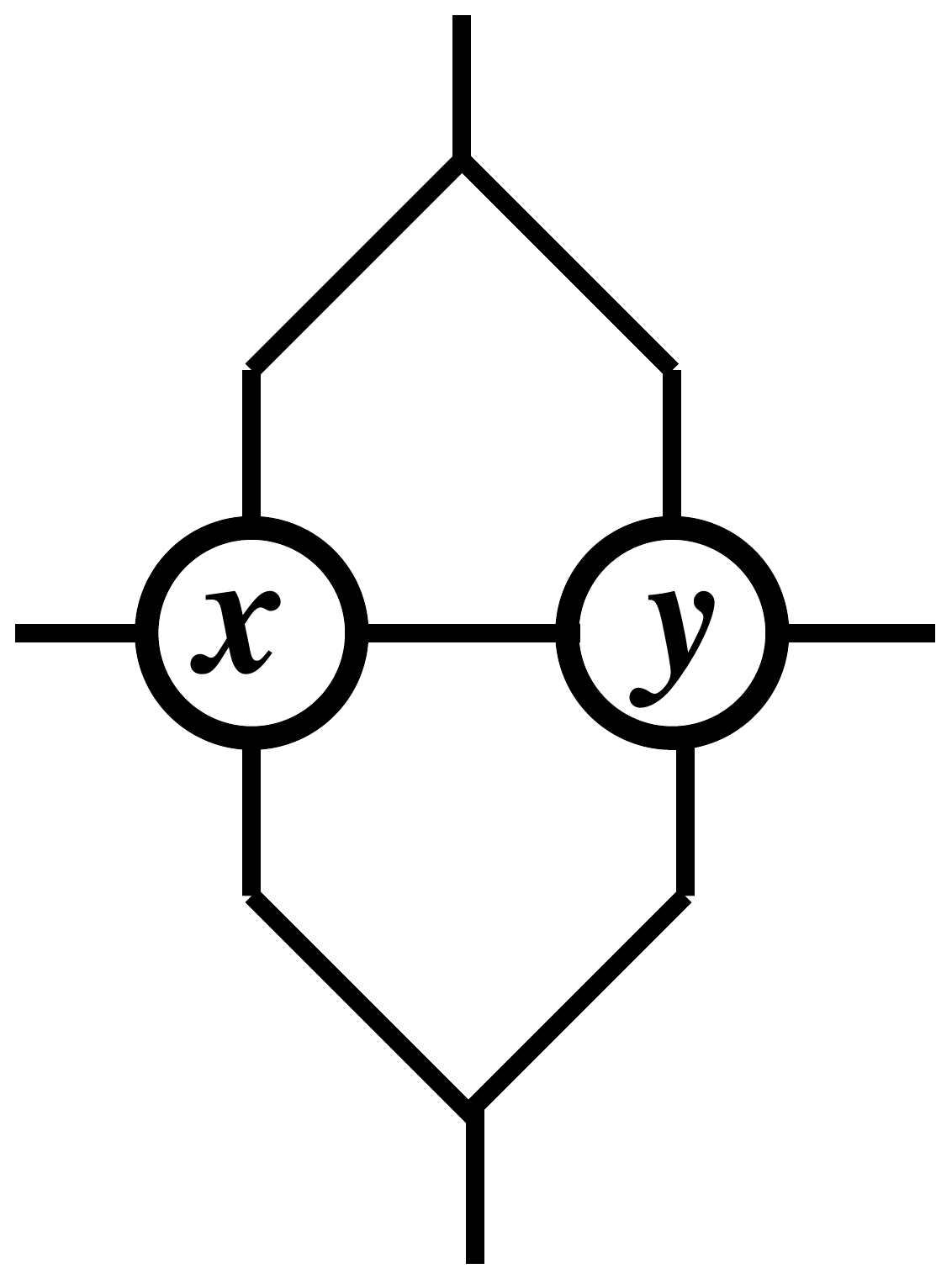} {1in}
and the renormalisation map $\mathscr{R}(x)=\mathscr{B}(x,x)$
\end{definition}

Observe that $\mathscr{B}$ makes $Q_4$ into a commutative non-associative algebra for which
$\mathscr{R}$ is the squaring operation.

We see the inner product formula becomes (if $y=\mathscr{R}(x)$):
\vskip 10pt
\hspace{1in} \vpic {nextround} {4in} 

Continuing in this way we see that 
$$\langle \rho_{\frac{1}{2^{k+n+1}}}\xi,\eta\rangle= \langle T_{2^k}(\mathscr{R}^{n}(x))\xi,\eta\rangle $$

We thus have to understand the iterates of the renormalisation transformation 
$\mathscr R:Q_4\rightarrow Q_4$. We begin by calculating $\mathscr R$ explicitly. For this we use the
basis  $\{\BA,\BB,\BC\}$ of $Q_4$ and write an arbitrary element of $Q_4$
as $$a=p\BA+q\BB+r\BC.$$
Since $\mathscr B$ is bilinear it is easy to expand and compute $\mathscr{R}(a)$ using the skein relations
in $\mathcal Q$. A sufficient set of relations is the following (see \cite{MPS}):

\vpic{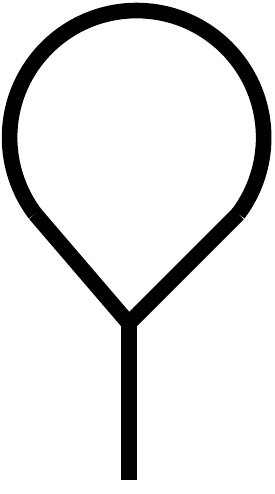} {0.1in} $=0$, \vpic{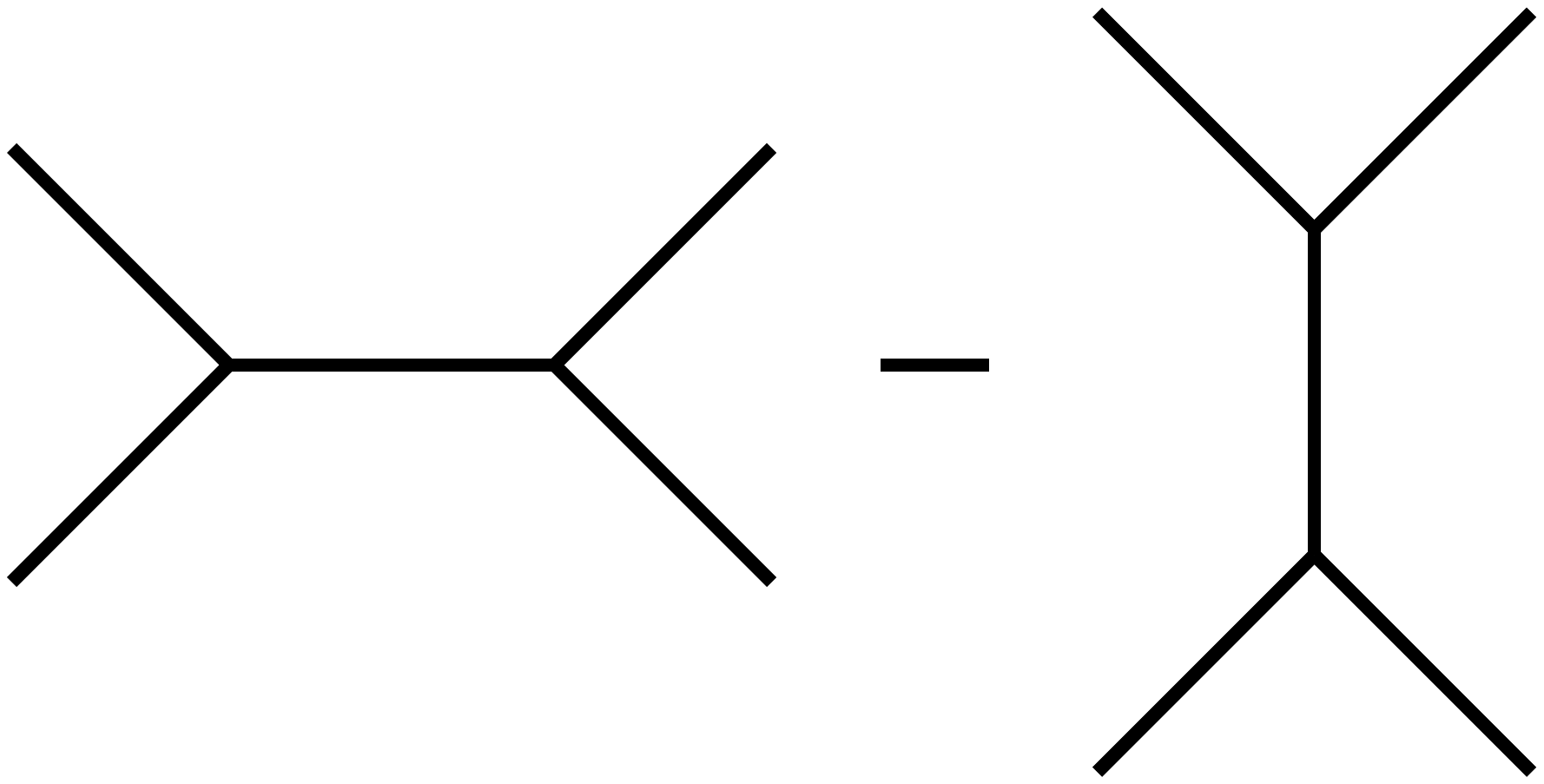} {0.5in} $=\frac{d-2}{d-1}( \vpic{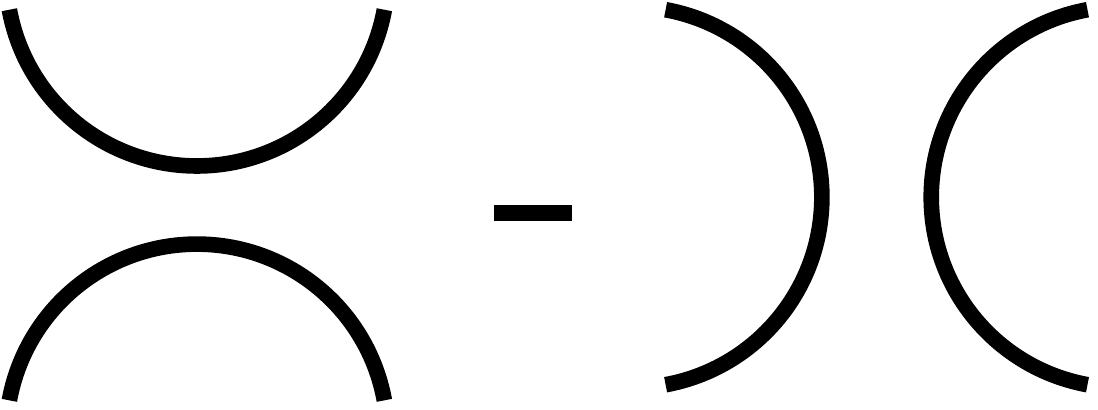} {0.5in} ) $ , and of course unitarity, \vpic{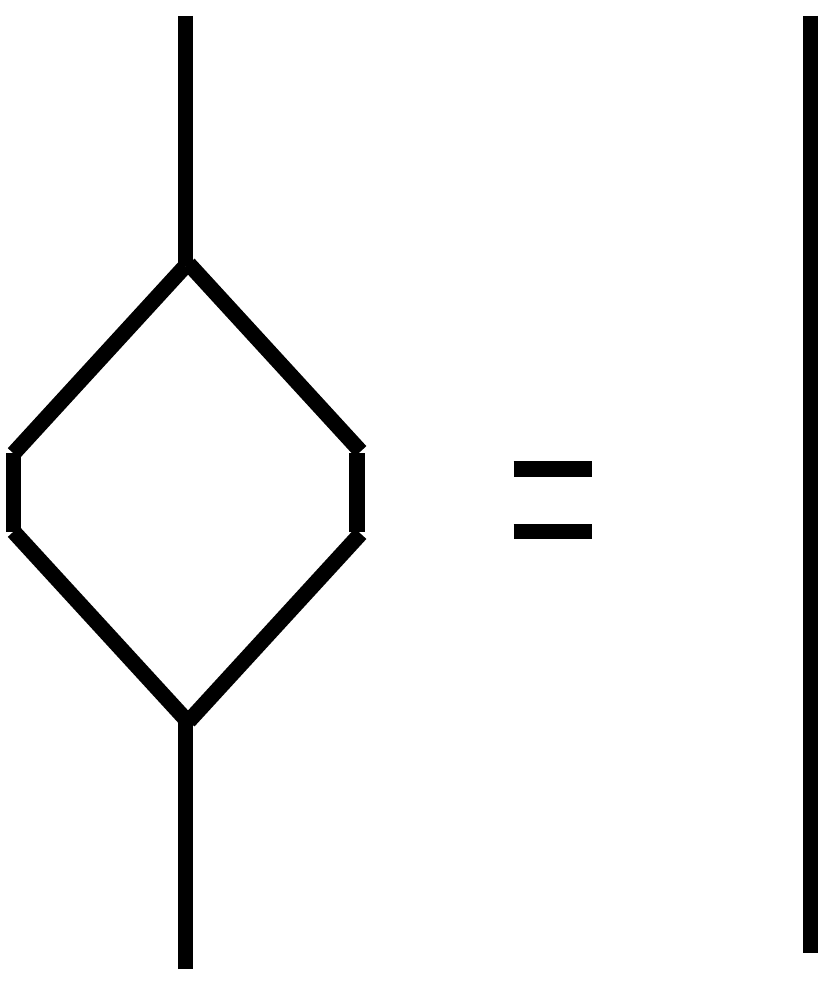} {0.3in} .

(A quick way to deduce the second picture-both sides are eigenvalues for the rotation of $\pi/2$ with eigenvalue $-1$.  But, 
modulo the TL subspace, the rotation has eigenvalue $+1$ by looking at a spanning set of TL diagrams reduced by the JW.
Thus the two sides of the equation are proportional and the constant can be obtained by capping and using unitarity.)

With these relations it is not hard to show that:\\
$\displaystyle \mathscr R(a)=\{\frac{d^2 - 5d + 7}{(d - 1)^2}p^2 + 2pq + 2\frac{d-2}{d-1}pr + q^2+r^2\}\BA\\ 
\mbox{                 }\qquad\qquad -\{\frac{1}{(d - 1)^3}p^2 + \frac{1}{d-1}(2pq + q^2)\}\qquad\BB$\\
\mbox{                 }\qquad\qquad  $\displaystyle +\{\frac{d^2 - 3d + 3}{(d - 1)^3}p^2 + \frac{1}{d-1}(2pq + q^2)\}\qquad\BC .$

 Completing some squares we get \\
 $\displaystyle \mathscr R(a)=\{(p+q)^2+(r+\frac{d-2}{d-1}p)^2-\frac{(d+1)(d-2)}{(d-1)^2}p^2\}\BA\\ 
\mbox{                 }\qquad\qquad -\{\frac{1}{(d - 1)}(p+q)^2- \frac{d(d-2)}{(d-1)^2}p^2\}\qquad\BB$\\
\mbox{                 }\qquad\qquad  $\displaystyle +\{(p+q)^2+(r+\frac{d-2}{d-1}p)^2- \frac{(d+1)(d-2)}{(d-1)^2}p^2\}\qquad\BC .$
 
Now define the norm $||-||_1$ 
 on $Q_4$ by $||p\BA+q\BB+r\BC||_1=|p|+|q|+|r|$ . Then the above shows that
$$||\mathscr R(a)||_1\leq \frac{d+1}{d-1}(p+q)^2+(r+\frac{d-2}{d-1}p)^2 +\frac{d(d+1)(d-2)}{(d-1)^3}p^2$$

 By convexity the maximum of the right hand side on the $||-||_1$ unit ball is 
 $$M=\frac{d+1}{d-1}+(\frac{d-2}{d-1})^2+\frac{d(d+1)(d-2)}{(d-1)^3}.$$
 Hence $$||\mathscr R(a)||_1\leq M||a||_1^2$$ and if there is an $n$ for which
 $\displaystyle ||\mathscr R^n(\BA)||_1<K $ for some $K>0$ with $MK<1$ then 
 $\displaystyle ||\mathscr R^{n+1}(\BA)||_1<KM||\mathscr R^n(\BA)||_1$ and 
  $\lim_{n\rightarrow \infty} {\mathscr R}^n(\BA)=0$. 
  
  Computer calculations show that such an $n$ exists (indeed is rather small) for all the values of $d$
 mentioned before the statement of the theorem.
 
 Now consider the following element  $Y\in Hom(V_{2^k+2},V_{2^k+2})$, with $y={\mathscr R}^n(\BA)$
(illustrated for $n=3$):\\
\mbox{  }\hspace{1.5in} $Y=$ \vpic {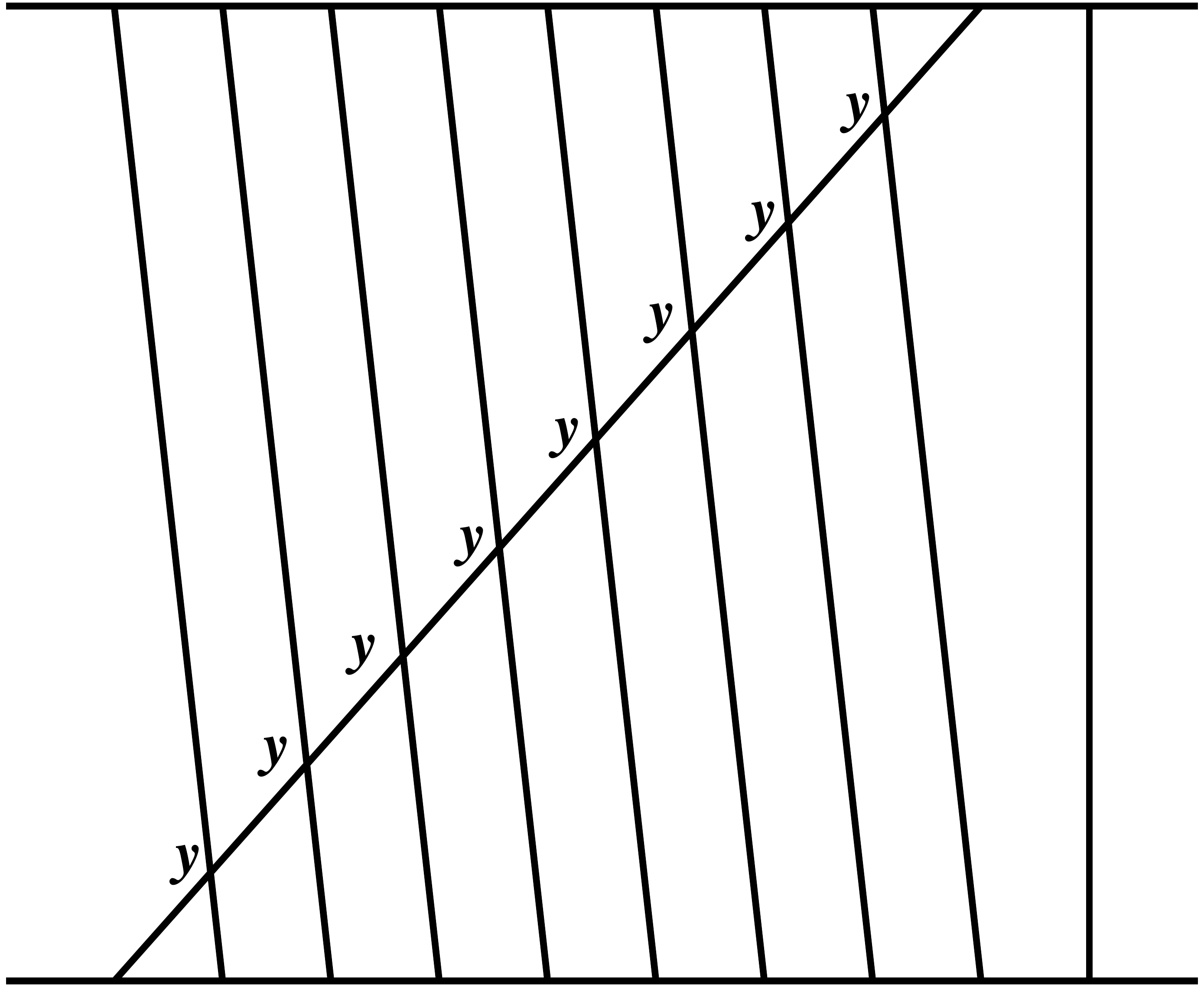} {1.7in} ,\\
and the following elements $\tilde \xi$ and $\tilde \eta$ of $V_{2^k+2}$:\\

$\tilde \xi =$ \vpic{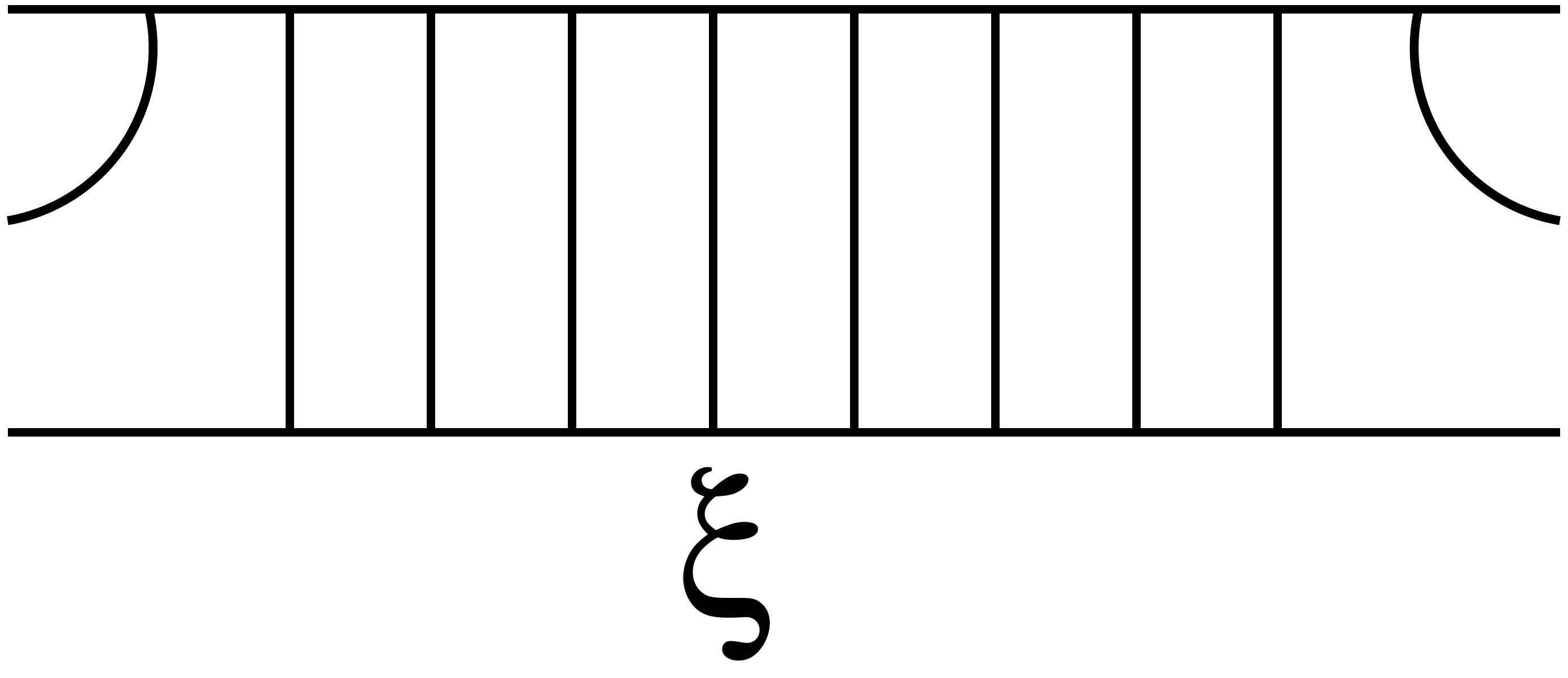} {1.5in}    \qquad $\tilde \eta =$ \vpic{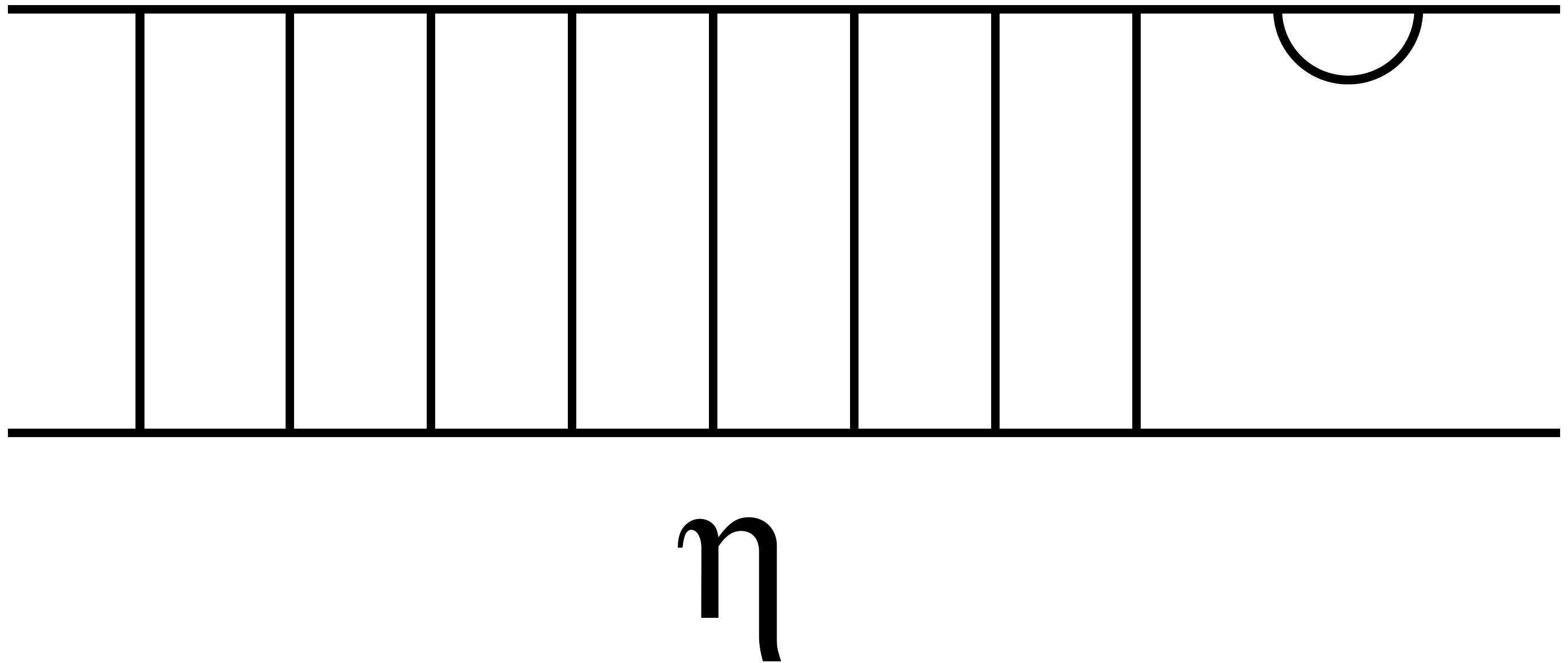} {1.5in}

(Remember that we have imposed periodic horizontal boundary conditions.)

Then a picture shows that $\langle Y\tilde \xi, \tilde \eta\rangle =\langle T_{2^k}(\mathscr{R}^{n}(x))\xi,\eta\rangle$

But we can now easily estimate $||Y||$ for it is a composition $y_1y_2\cdots y_n$ where $y_i$ is the element of 
$Hom(V_{2^k+2},V_{2^k+2})$ with a copy of $y$ between the $(i+1)th.$ and $(i+2)th.$ boundary points as illustrated
below:\\

\mbox{   } \qquad $y_i=$ \vpic{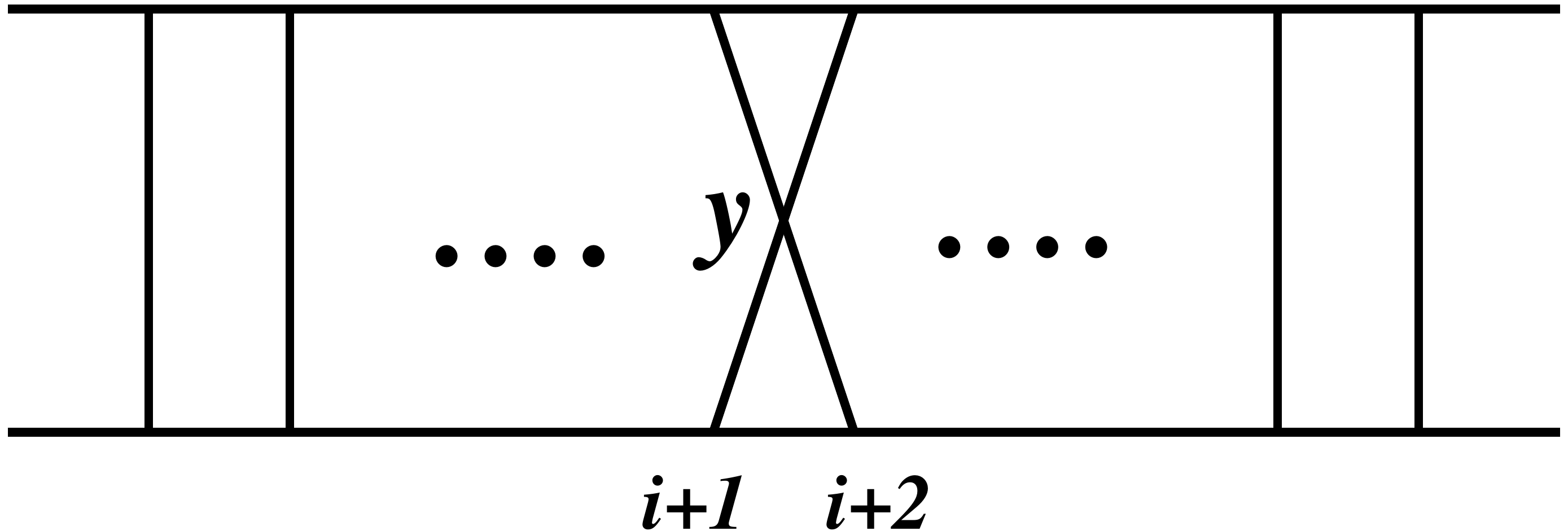} {2.3in}

But the norms of the $y_i$ are all equal to the norm of $y$ as an element of the C$^*$-algebra $Q_4$. 
And, all norms being equivalent, we have shown that $||Y||\rightarrow 0$ as $n\rightarrow \infty$.
This proves the theorem.

\end{proof}
  
\begin{appendix}
\section{Some notions of planar algebra.} \label{planar}

In this paper a planar algebra $\mathcal P$ will be a graded vector space $P_n$, graded  by $\mathbb N\cup \{0\}$ and admitting multilinear operations indexed
by \emph{planar tangles} $T$ which are subsets of the plane consisting of a large (output) circle containing smaller (input circles). There are also 
non-intersecting smooth curves called strings whose end points, if they have any, lie on the  circles where they are called marked points. Elements of $\mathcal P$ are ``inserted'' into 
the input circles with an element of $P_n$ going into a disc with $n$ marked points, and the result of the operation specified 
by the tangle is in $P_k$ where
there are $k$ marked points on the output circle. In order to resolve cyclic ambiguities, each of the circles of $T$ comes with
a privileged interval between marked points which we will denote in pictures by putting a $\$ $ sign near that interval. The \$ signs
are used to define an obvious notion of gluing of one tangle inside an internal disc of another.

Here is an example of a planar tangle: \\

\hspace{1in} \hpic{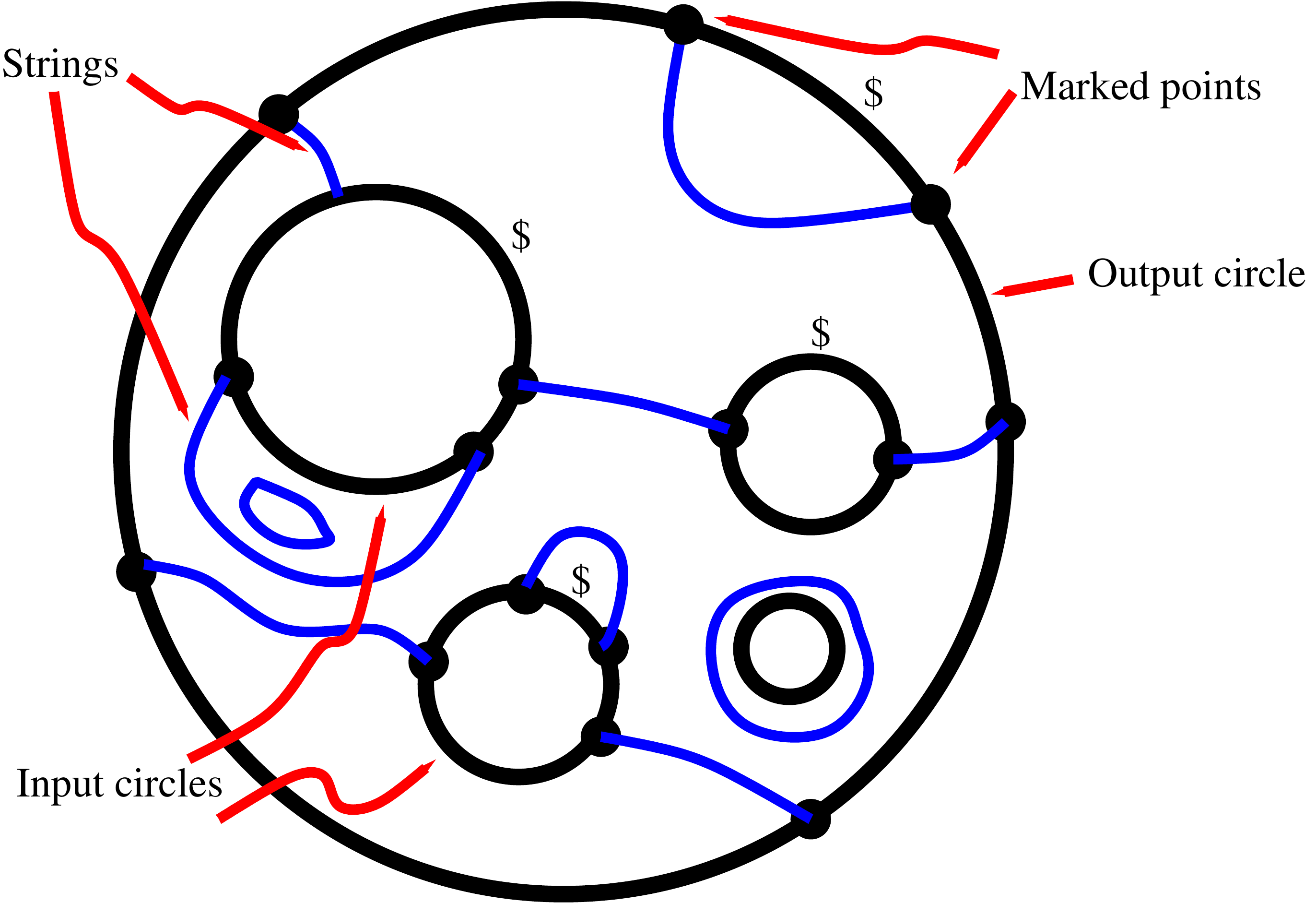} {2in}

The result of the operation indexed by $T$ on elements $v_1,v_2,\cdots, v_n$ of $\mathcal P$ is denoted $Z_T(v_1,v_2,\cdots, v_n)$ where there are $n$ input discs. See \cite{jo2} for details.
The operation $Z_T$ depends only on $T$ up to smooth planar isotopy so one has a lot of freedom drawing the tangles, in particular the circles may be replaced
by rectangles when it is convenient.  The operations $Z_T$ are compatible with the gluing of tangles.
Tangles may also be ``labelled'' by actually writing appropriately graded elements of $\mathcal P$ inside some of the internal circles.

\emph{It is a very useful convention to shrink the input discs in a planar tangle to points so that the boundary intervals of the circle become
the regions adjacent to the points. And for labelled tangles one places the label in the region corresponding to the $\$$ sign. }Thus\\

\vpic{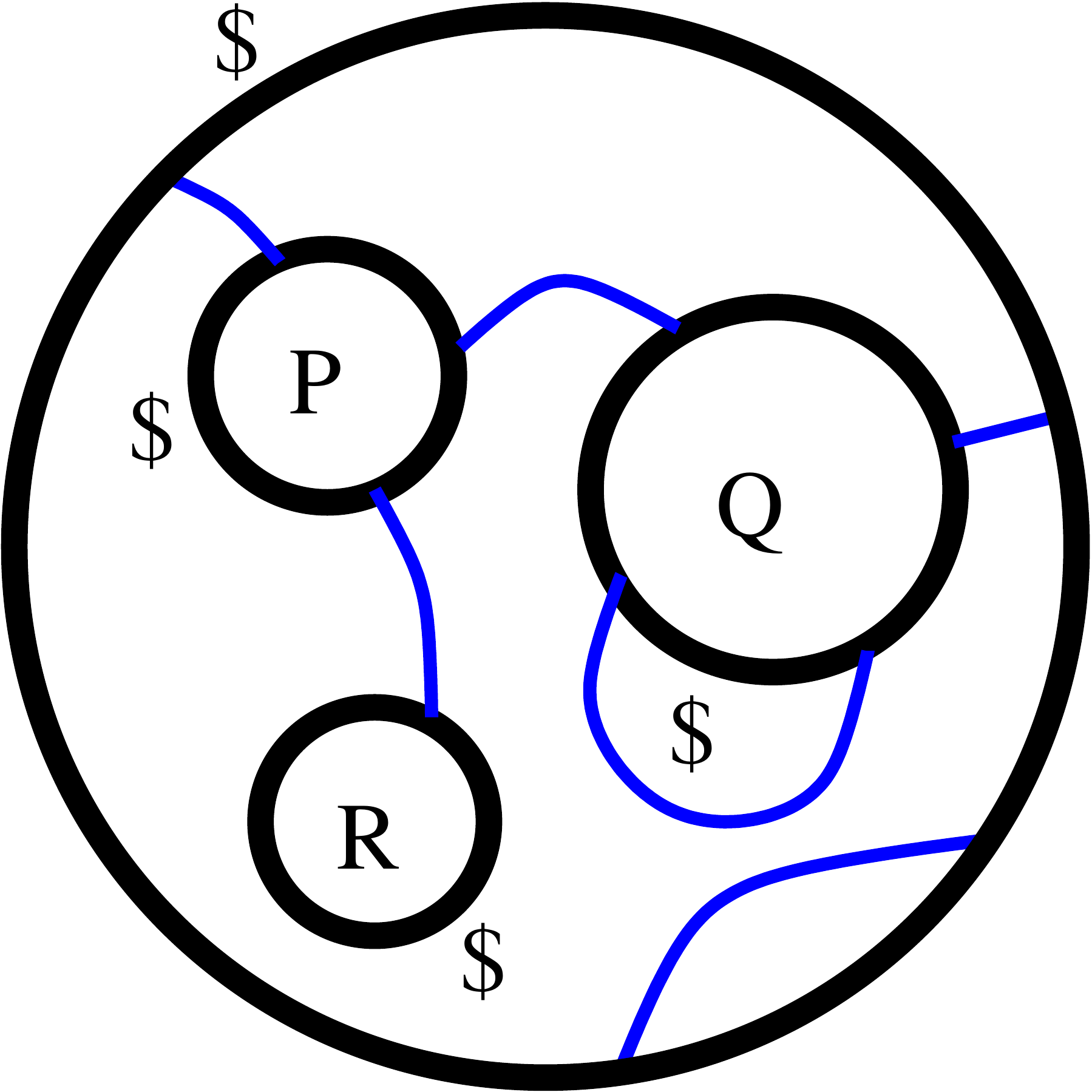} {1in} is represented by the picture \vpic{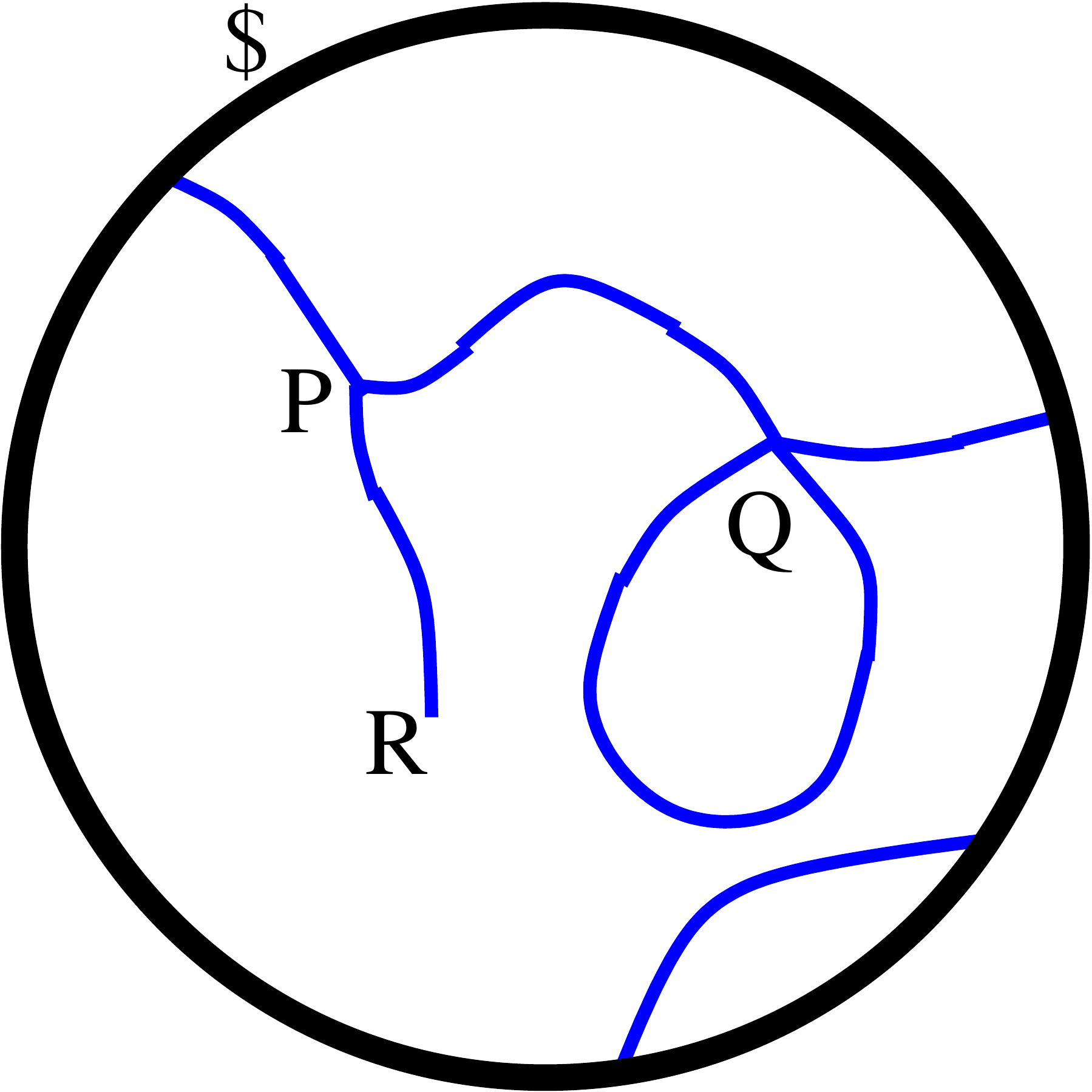} {1in}

We will also often omit the output disc and/or dollar signs provided they are obvious in context.

\begin{definition} \label{partition function}Given a  planar tangle $T$ all of whose internal circles are labelled by  $v_1,v_2,\cdots, v_n$
 we call  $Z_T(v_1,v_2,\cdots, v_n)$
 the element of $P_k$ which it defines. If $k=0$ and the dimension of $P_{0}$ is one, this may be identified
 with a scalar using the rule that $Z(\rm{empty tangle})=1$.
 \end{definition}

Planar tangles can be glued in an obvious way along input circles and  the operations $Z_T$ are by definition compatible with
the gluing.

 For connections with physics and von Neumann algebras, planar algebras will have more structure, namely an antilinear involution $*$ 
 on each $P_n$ compatible with orientation reversing diffeomorphisms acting on tangles.  If, moreover, $dim P_0=0$ we get a sesquilinear
 inner product $\langle S, R\rangle$ on each $P_n$ given by \vpic {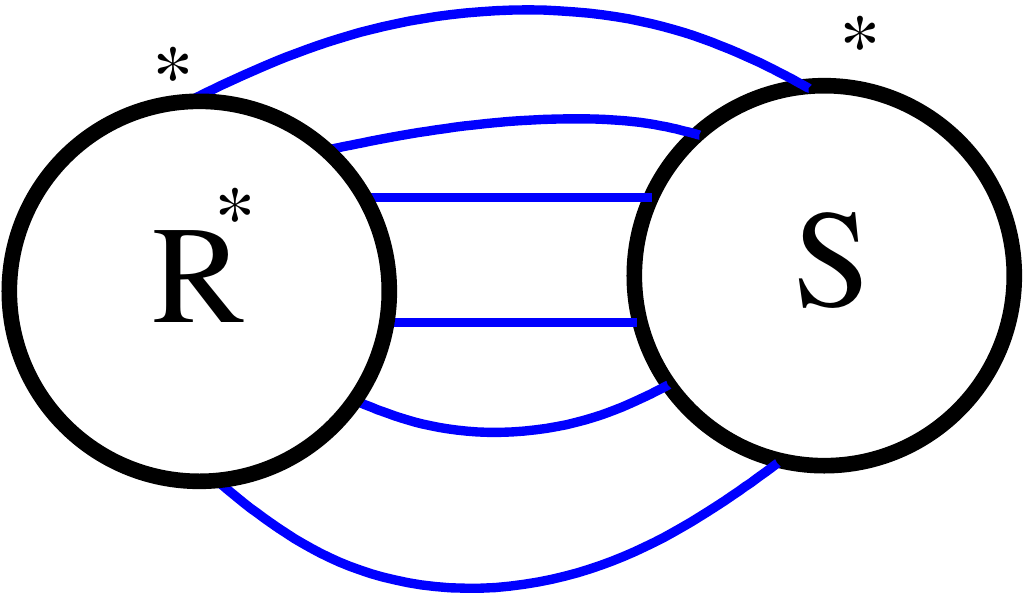} {1in}. 
 
 A planar algebra will be called \emph{positive definite} if this inner product is. 
 
  Our planar algebras will all have a parameter $\delta$ which is the 
value of a closed string which may be removed from any tangle with multiplication by
the scalar $\delta$.

Two examples of planar algebras should be mentioned. The first is the Temperley-Lieb algebra  $TL$ (which has its origins
in \cite{TL} though its appearance here should properly be attributed to \cite{Kff2}, via \cite{Ba2}-see also \cite{J9}).  A basis of $TL_n$ consists of all isotopy classes of systems of non-crossing strings joining
$2n$ points inside the disc. In particular $TL_n$ is zero if $n$ is odd. The planar algebra operations are the obvious gluing ones with the rule that any
closed strings that may be formed in the gluing process are discarded but each one counts for a multiplicative
factor of $\delta$, called the ``loop parameter".  The * structure is given by complex conjugation on 
basis diagrams, extended by conjugate linearity. This planar algebra is positive definite iff $\delta \geq 2$. If $\delta=2cos\pi/n$ for
$n=3,4,5,\cdots$ $TL$ admits a quotient planar algebra which is positive definite. 


The second examples of  planar algebras which we will use are the tensor planar algebras. For fixed integer $k\geq 2$ one considers
A Hilbert space $V$ of dimension $k$ with a basis  so that elements of the tensor power $\otimes^n V$ may be represented
as tensors with $n$ indices, each index running from $1$ to $n$.
 The  planar algebra $\mathcal P^{\otimes}$ is then defined by 
$P^{\otimes}_0=\mathbb C$, and for $n\geq 1$, $P^{\otimes}_n=\otimes^n V$ .
The action of planar tangles on tensors is nothing but contraction of tensors along the indices connected by the the strings
of the tangle, together with the rule that indices have to be constant along the strings.
The tensor planar algebras $\mathcal P^{\otimes}$ are positive definite when given the *-structure
 $$R^*_{i_1,i_2,\cdots,i_n}=\bar R_{i_n,i_{n-1},\cdots,i_1}.$$


\section{The affine category of a planar algebra.}\label{affinecat}
\begin{definition} The \emph{affine} category $Aff(\mathcal P) $ is the (linear) category whose objects are sets $\bar m$ of $m$ points
on the unit circle in $\mathbb C$ ,and whose vector space of  morphisms from $\bar m$ to $\bar n$ is the set of linear combinations of labelled tangles 
(with marked boundary points $\bar m\cup \bar n$) between the unit circle and a circle of
larger radius  modulo any relations in $\mathcal P$ 
which occur within contractible discs between the unit circle and the larger circle. Composition of morphisms comes from rescaling and gluing 
the larger circle of the first morphism to the smaller circle of the second.

\end{definition}

If $P$ is positive definite the morphism spaces of  $Aff(\mathcal P) $ admit an adjoint $x\mapsto x^*$ obtained by reflecting a labelled annular tangle about a circle
between the inner and outer circles of the tangle and taking the $*$'s of the labels.
 
Use of $\bar m$ adds to clutter so we will abuse notation by using just $m$ for an object of $Aff(\mathcal P) $ with $m$ points. We could also just suppose that 
the boundary points are always just the roots of unity.

One needs to be careful with this definition (see \cite{jo3},\cite{gl}). In a representation of $Aff(\mathcal P) $, morphisms may be changed by planar  isotopies 
without affecting the action, but the
isotopies are required to be the identity on the inner and outer circles.   Thus the tangle of rotation by 360 degrees does not necessarily act by the
identity in a representation of $Aff(\mathcal P)$.

The representations we will consider of $Aff(\mathcal P)$ are called lowest weight modules and may be defined as in \cite{jo3} by taking a representation $W$
of the algebra $Mor(n,n)$ for some $n$ (the ``lowest weight'') and inducing it in the obvious way. This may cause problems with positive definiteness but it is known that
subfactor planar algebras possess a host of such representations. The vector spaces $V_k$ of such a lowest weight representation  are zero 
if $k\leq n$ and spanned by  diagrams consisting of a vector $w\in W$ inside a disc with $n$ marked points, surrounded by a labelled planar tangle
of $\mathcal P$ with $k$ marked points on the output circle. 
 
 Here is a vector $w$ in  a $V_6$ created by the action of an affine morphism on $v$ in the lowest weight space $V_2$: 
  and the action of a morphism in $Aff(\mathcal P)$ on it:\\
 
 \hpic{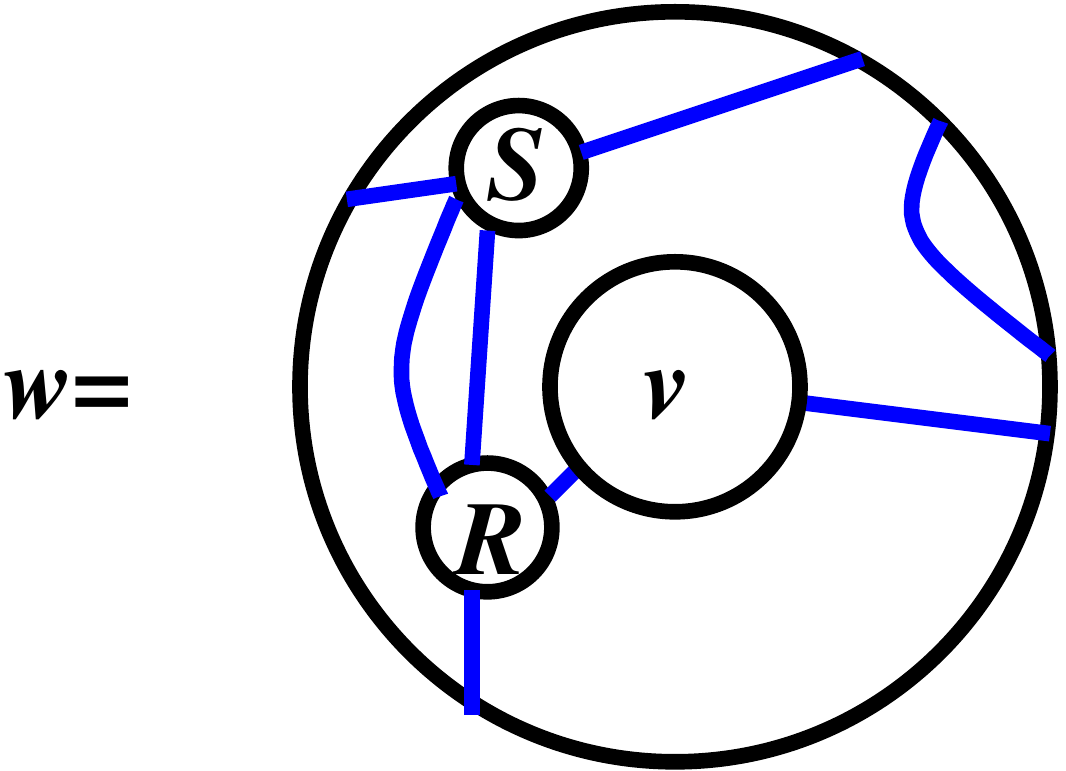} {1in}

 and  here is a diagram illustrating the result of acting on the above vector $w$ with   a morphism in $Mor(6,4)$:\\
 
 \hpic{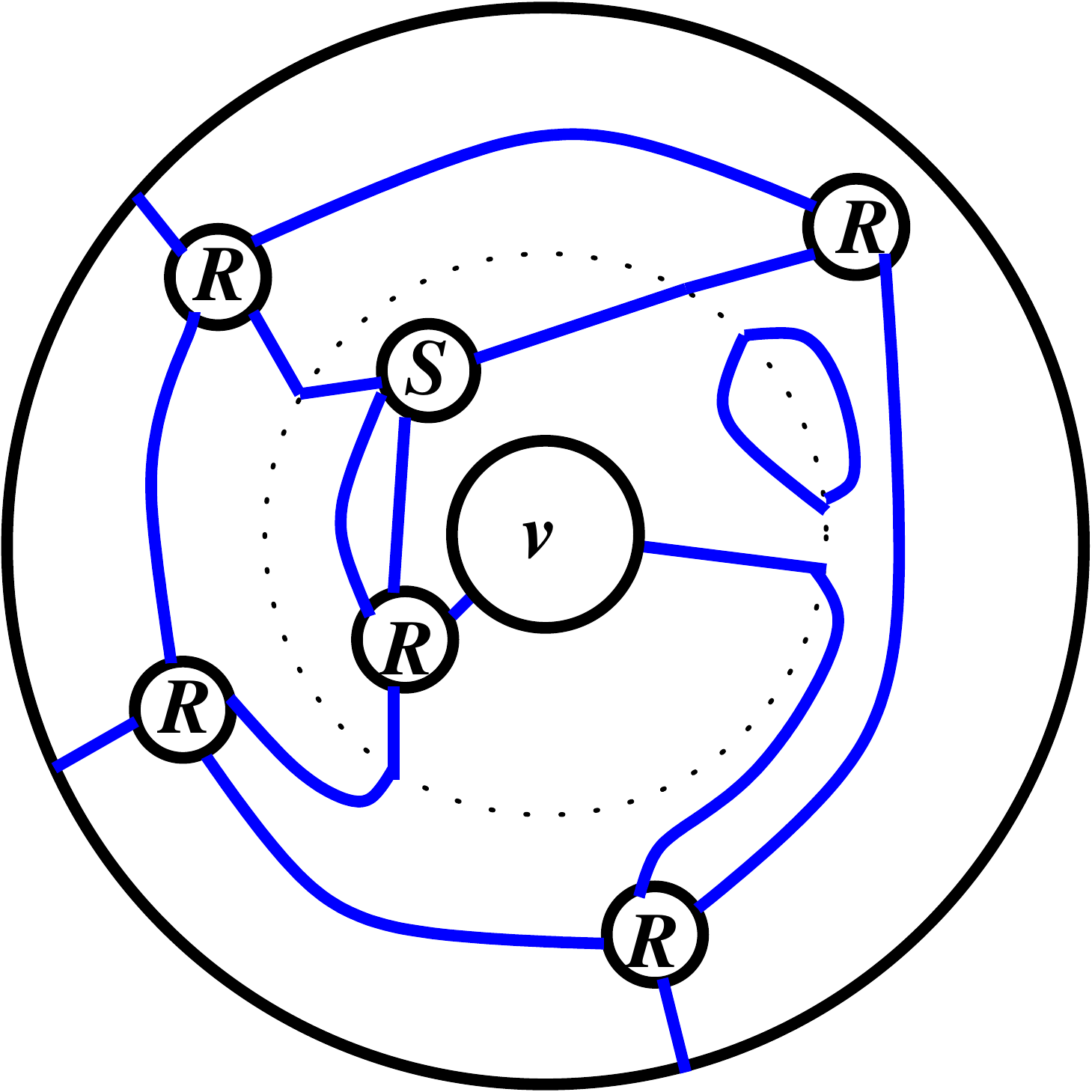} {1.5in}

 The planar algebra itself defines an affine representation simply by applying annular labelled
tangles to elements of $\mathcal P$. This representation  is irreducible  and plays the role of the trivial representation.

In the TL case which is what we will mostly consider, irreducible lowest weight representations are parametrized by their lowest weight (the smallest $n$ 
for which $V_n$ is non-zero), and a complex number of absolute value one which is the eigenvalue for the rotation tangle.
The case $n=0$ is exceptional and the rotation is replaced by the tangle which surrounds an element $v$ of $V_0$ by a circular string. If $v$ is an eigenvector for this tangle and there are some restrictions on the eigenvalue $\mu$-see \cite{jo3},\cite{JR}.
The case where $\mu=\delta$ is precisely the trivial representation. In this case the vector $v\in V_0$ is the empty diagram so it never features
in pictures. 

\begin{definition} An affine representation $V=V_{k}$ of a positive definite planar algebra will be called a \emph{Hilbert} representation
if each $V_k$ is equipped with a Hilbert space inner product which is invariant in the sense that $\langle a\xi,\eta\rangle = \langle \xi, a^*\eta\rangle$ where $a\in Mor(m,n), \xi\in V_n$, $\eta\in V_m$ and $*$ is the structure we defined earlier on the affine category.
 
\end{definition}

There are two particularly important affine tangles which  play a big role in our main theorem.
\begin{definition}
If $S\in P_4$ we define the ``transfer matrix'' $T_n(S)$ to be the element of  $Mor(n,n)$ defined by the following annular tangle:

\hspace {1in} \vpic{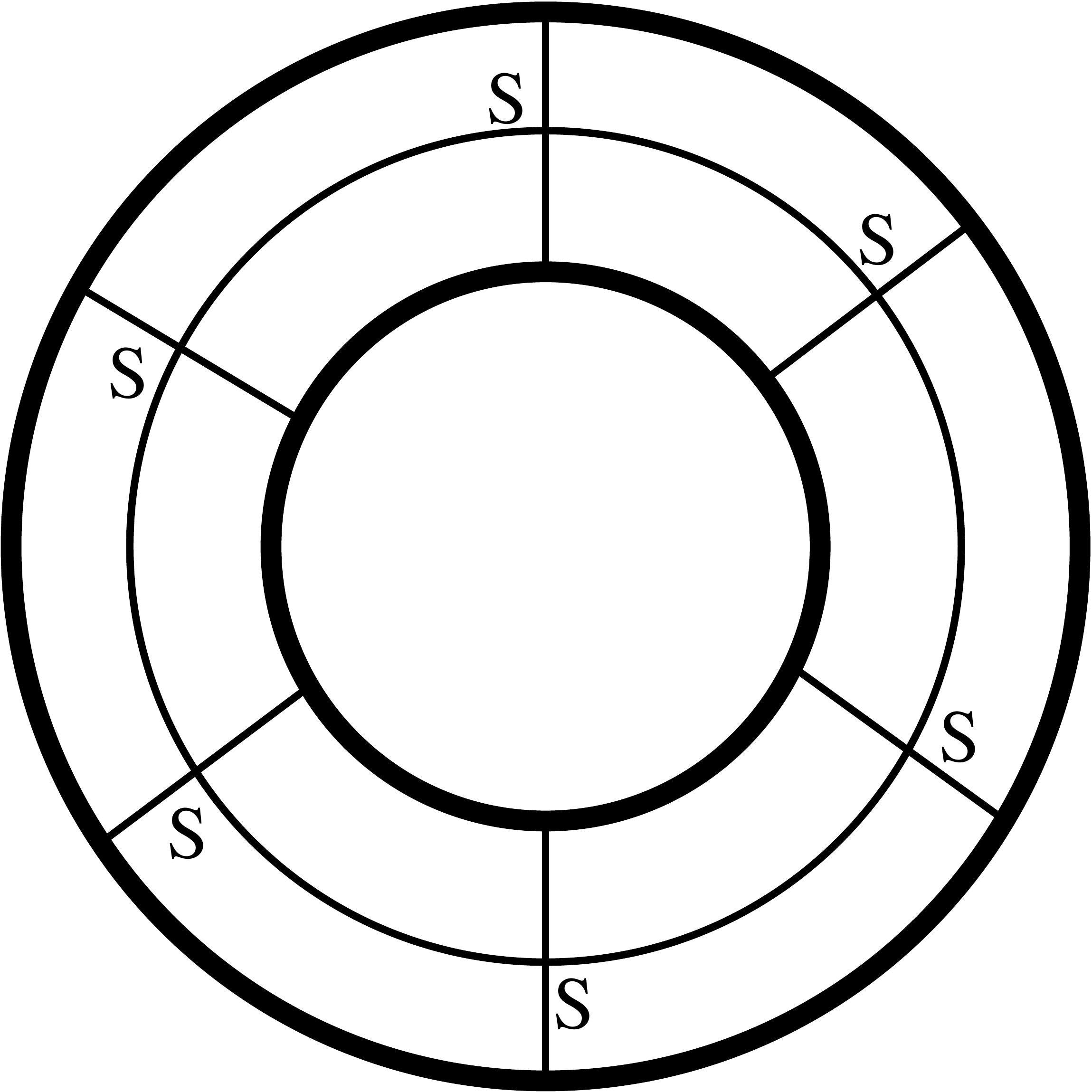} {1.5in} 

which it will be more convenient to draw horizontally with implcit periodic boundary conditions thus:\\

\hspace {1in} \vpic{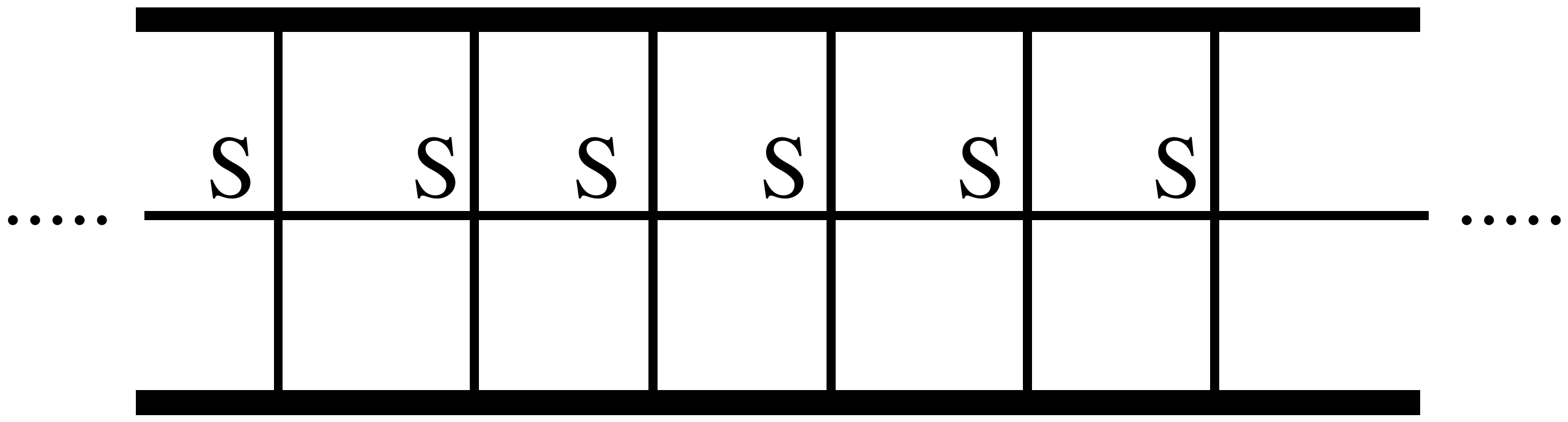} {2.3in}
 \end{definition}
(Here we have illustrated with $n=6$.) 
The second important tangle is the rotation
\begin{definition} The rotation $\rho_n \in Mor(n,n)$ is given by:\\
\mbox{   } \hspace {1in} \vpic {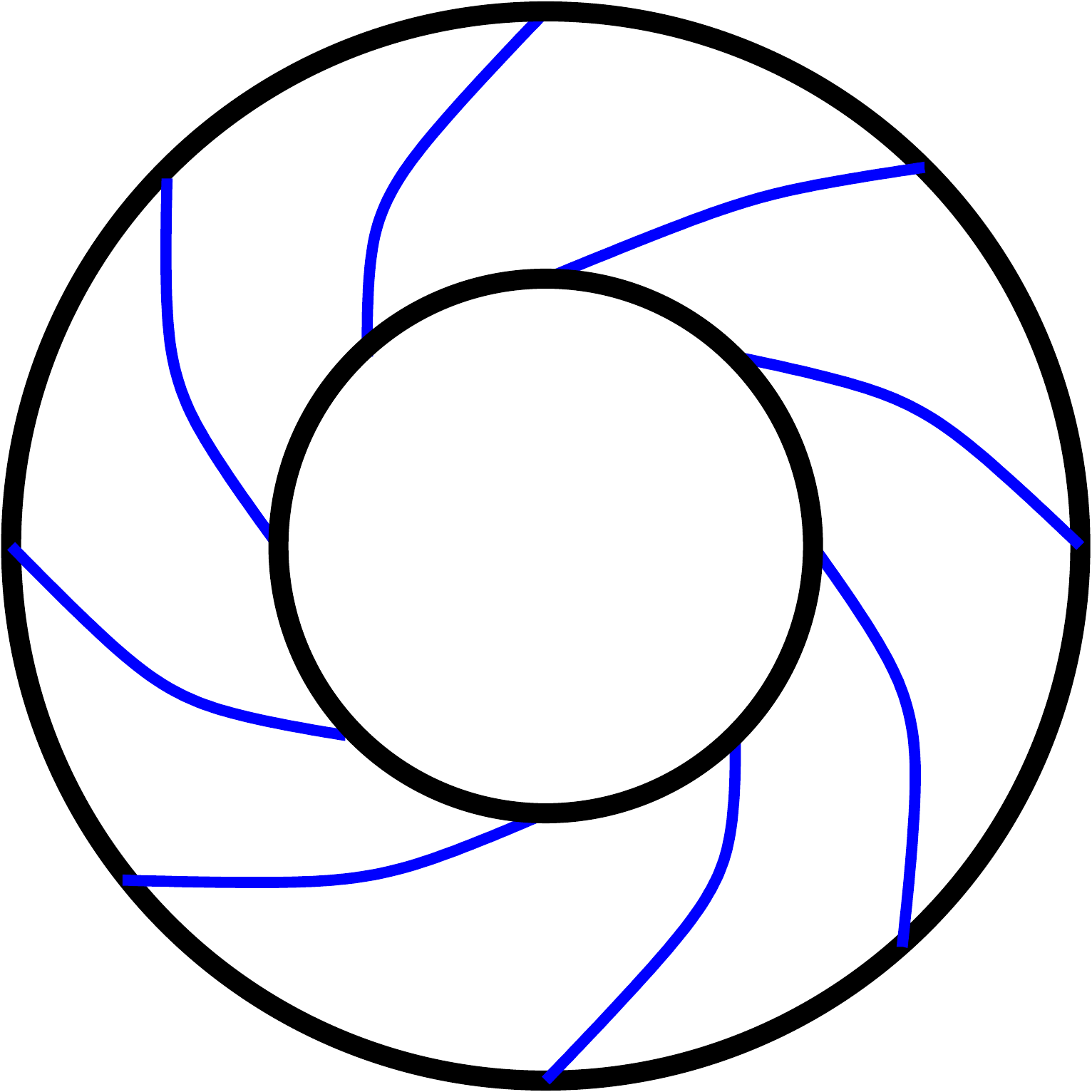} {1.5in}

(illustrated for $n=8$).
\end{definition}

\begin{definition}\label{annular} An affine representation $V=V_{k}$ of a  planar algebra will be called an \emph{annular} representation
if the rotation by $2\pi$ ($=\rho_n^n$) acts by the identity. \end{definition}

$\rho_n$  generates a copy of $\mathbb Z$ inside $Mor(n,n)$.

\end{appendix}

\end{document}